\documentclass[11pt,reqno]{amsart}
\usepackage{latexsym,amssymb,amsmath}
\usepackage{eufrak,fullpage}
\usepackage{graphics,color}
\usepackage{graphicx}

\usepackage{textcomp}

\usepackage{epsfig}
\usepackage[applemac]{inputenc}
\usepackage{subfigure}
\usepackage{epstopdf}
\usepackage{url}
\usepackage{hyperref}

\hypersetup{
	colorlinks,
	citecolor=black,
	urlcolor=black,
	linkcolor=black,
	filecolor=black}

\renewenvironment{proof}{{\noindent\bfseries Proof.}}{\qed}

\newtheorem{conjecture}{Conjecture}%[section]
\newtheorem{proposition}{Proposition}%[section]
\title{Partially Directed Snake Polyominoes}

\author{Alain Goupil,  Marie-Eve Pellerin, J\'er\^ome de Wouters d'Oplinter} %
\address{ D\'{e}partement de math\'{e}matiques et d'informatique, Universit\'{e} du Qu\'{e}bec \`{a} Trois-Rivi\`{e}res, Trois-Rivi\`{e}res (QC) Canada} 
\email{alain.goupil@uqtr.ca, Marie-Eve.Pellerin@uqtr.ca,jeromedewouters@gmail.com}
\begin{document}

\maketitle
\begin{abstract}
The goal of this paper is to study the family of snake polyominoes. 
More precisely,  we focus our attention on the class of partially directed snakes. We establish functional equations and length generating functions of two dimensional, three dimensional and then $N$ dimensional  partially directed snake polyominoes.  We then turn our attention to  partially directed snakes inscribed in a $b\times k$ rectangle and we establish two-variable generating functions, with respect to  height $k$ and  length $n$ of the snakes.  We include observations on the relationship between snake polyominoes and self-avoiding walks.
We conclude with a discussion on inscribed snakes polyominoes of maximal length which lead us to the formulation of a conjecture encountered in the course of our investigations. 
\end{abstract}

%=============================================================
% \section  Introduction
%=============================================================

\section{Introduction} 

A two dimensional ($2D$) polyomino is an edge-connected set of cells in a regular lattice up to translation. In this paper we are solely  interested with sets of square cells with sides of unit length in the square lattice. Since a polyomino can be seen as a simple graph with its  square cells as vertices and  its edge contacts between  adjacent squares as edges, we will  use the  vocabulary of graph theory when we find it convenient. 
Thus the degree of a cell in a polyomino is the number of its edge contacts with its neighbour cells. 

A {\em snake polyomino} of length $n$, shortly referred to as a snake,  is a polyomino with $n$ cells, no cycle and such that each cell has edge degree at most two. As a consequence a snake of length at least two has two cells of degree one.  The two cells of degree one, called the tail and the head, can be used to provide snakes with an orientation. An {\em oriented snake} of length $n\geq 2$ is a snake with  designated tail and  head cells and the orientation goes from tail to head. When  the tail and head cells are not specified in a snake, we say that the snake is {\em non-oriented}. A snake can always be given an orientation and we will call upon this property when needed but snakes are a priori  considered non-oriented. 

Oriented snake polyominoes are similar objects to self-avoiding walks (see~\cite{Bm}) but, to our knowledge, no bijective relation between the two sets has been established so far and we found no systematic investigation of snakes in the current literature. 

It is in fact quite straightforward to see that  there is an injective map $\phi$ from the set $OS(n)$ of oriented snakes of length $n$ to the set $SAW(n)$ of self-avoiding walks with $n$ vertices  by simply transforming cells of snakes into vertices of walks and drawing an edge between two vertices when the two corresponding cells are edge connected (Figure~\ref{inj}).  This means that the set $OS(n)$ is in bijection with a subset of $SAW(n)$.  This correspondence is not invertible because the minimum number of cells needed to form a $U$ shape in a snake is superior by one  to the minimum number of vertices in a $U$ shape of a $SAW$.  By a $U$ shape, we mean either a snake or a self-avoiding walk that goes in a given direction, say East, then moves  in a perpendicular direction, say North, then returns in the direction opposite to the initial one, here West, thus forming a $U$ shape. Of course the initial direction in a $U$ shape can be any of the four directions.

\begin{figure}[htbp]
\label{inj}
\begin{center}
\epsfig{file=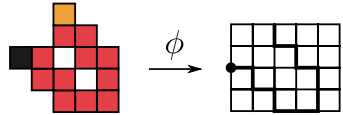,scale=0.8}
\caption{An injective map from oriented snakes to self-avoiding walks of same length}
\label{classific}
\end{center}
\end{figure}

A polyomino $P$ is said to be inscribed in a rectangle $R$ when $P$ is included in $R$ and  the cells of $P$ touch each of the four sides of $R$.  Enumeration of inscribed polyominoes with area minimal, minimal plus one and minimal plus two were investigated in~\cite{GCN,GCP}.  Inscribed $3D$ polyominoes of minimal volume  were investigated in~\cite{GC}.  Here the  notation $b\times k$  is used for the dimension of the circumscribed rectangle.  This paper is devoted to the enumeration of partially directed snakes but more generally, we are interested in the combinatorial study of different classes of snake polyominoes which are presented in section~\ref{class}. 

In section~\ref{nonins} we construct the length generating functions of partially directed snakes. We start with the two dimensional case, then we show that it is possible to extend the construction to the three dimensional and more generally to the $N$ dimensional cases. In section~\ref{insc}  we construct,  for each value of the width $b$ of the circumscribed rectangle, a two-variable generating function for inscribed partially directed snakes. The two parameters are the length $n$ of the snake and the height $k$ of the rectangle. In order to realize this construction, we define  subclasses of partially directed snakes and we establish their generating functions. In particular we define the class of bubble snakes which is related to bargraph polyominoes (see~\cite{BmR}) and we establish their length generating function by way of a combinatorial factorization. In section~\ref{bij} we describe a straightforward  bijection relating bargraphs to snakes  that we have not seen elsewhere in the literature.  In the final section, we present a conjecture concerning inscribed snakes of maximal length.

%=============================================================
% \section{Classification of snakes} 
%=============================================================

\section{Families of snakes} 
\label{class}
Our classification of snake polyominoes follow closely that of self-avoiding walks presented in~\cite{Bm} from which we borrow part of the terminology. Nevertheless snake polyominoes possess one geometrical property that cannot exist in self-avoiding walks: two cells of a snake separated by at least two cells may touch with their corners.  
When there is no such pair of cells in a snake, we say that we have a {\em kiss-free} snake.  

A snake polyomino is a {\em partially directed} snake ({\em PDS}) towards the North when, starting from the tail, each new cell is added in one of the three directions North, East or West with the necessary constraint that East cannot follow immediately West and vice versa. As a consequence, when a west step follows an east step in a {\em PDS}, there are at least two north steps  between them. The same is true when East follows West. 

Throughout this paper, we will adopt the convention of using capital letters $F(s,t,q)$ for generating functions, lower case letters $f(b,k,n)$ for their coefficients and calligraphic capital letters $\mathcal{F}(b,k,n)$ for the corresponding sets of snakes of length $n$ inscribed in a $b\times k$ rectangle:
\begin{align*}
F(s,t,q)=\sum_{b,k,n}f(b,k,n)s^bt^kq^n.
\end{align*}
The first values of $s(n)$, the total number of non-oriented snakes of length $n$, are presented in table~\ref{tab2} (see {\em OEIS} A182644 (voir \cite{OEIS}) for $n\leq 24$).  

\begin{table}[htdp]
\begin{center}
\scalebox{0.9}{
\begin{tabular}{|c|c|c|c|c|c|c|c|c|c|c|c|c|c|c|c|c|}
\hline
$n$&1&2&3&4&5&6&7&8&9&10&11&12&13&14&15&16\\
\hline
$s(n)$&1&2&6&14&34&82&198&470&1122&2662&6334&14\ 970&35\ 506&83\ 734&198\ 086&466\ 314
\\
\hline
\end{tabular}}
\end{center}
\caption{Number of snakes of length $n$}
\label{tab2}
\end{table}%

%=============================================================
% \section{Non-inscribed partially directed snakes} 
%=============================================================

\section{Non-inscribed partially directed snakes}
\label{nonins}
When we want to emphasize on the fact that the rectangular dimensions of a snake are not constrained to specific values, we say that it is  {\em non-inscribed} but when no confusion is possible, we simply call it a snake. 
The parameter of interest for these snakes is the length. 
In this section three families of  {\em PDS} are investigated:  two dimensional, three dimensional and $N$ dimensional snakes. 
For  each family, a functional equation is constructed and solved to obtain a one-variable rational generating function. Exact formulas in terms of the length of the snakes are presented for  the $2${\em D} and the $3${\em D} cases.

\subsection{Two dimensions}
 
The one-variable generating function of $2${\em D}  {\em PDS}, denoted $PDS_{2D}(q)$, is defined by
\begin{align*}
PDS_{2D}(q)=\sum_{n\geq 0} pds_{2D}(n) q^n
\end{align*}
where $pds_{2D}(n)$ is the number of {\em PDS} of length $n$.

Recall that a {\em PDS} is formed using north, east and west steps only.  Moreover two north steps are required to change direction between east and west. Thus the construction of a functional equation for {\em PDS} uses similar ideas to that of partially directed {\em SAW} (see~\cite{Bm} or~\cite{St} for generating function), also known as one-sided {\em SAW}, that have to be adapted to these additionnal constraints. 

In order to build the functional equation, note that the set of {\em PDS} admits a partition into two sets:
\begin{itemize}
\item {\em PDS} without two consecutive north steps, $\mathcal S^1_{2D}$;
\item {\em PDS} containing two consecutive north steps at least once, $ \mathcal S^2_{2D}$.
\end{itemize}
Samples of each sets are illustrated in Figure~\ref{NonInsc_pds1N2N}.  
\begin{figure}[h!]
\centering
\subfigure[$\mathcal S^1_{2D}$]{{\includegraphics[scale=0.8]{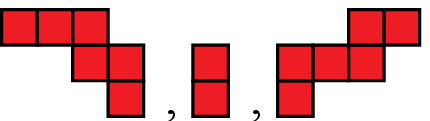}}}\qquad\qquad
\subfigure[$ \mathcal S^2_{2D}$]{{\includegraphics[scale=0.8]{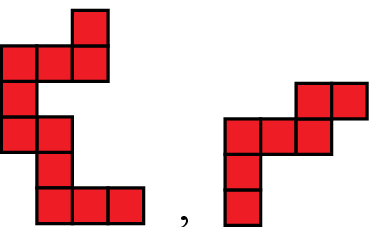}}}
\caption{A partition of the set of  {\em PDS} \label{NonInsc_pds1N2N}}
\end{figure}

Since two consecutive north steps are required to change direction from East to West or from West to East in a {\em PDS}, snakes in $\mathcal S^1_{2D}$ cannot have both west and east steps.  In fact these snakes either have north or east steps only or they have north or west steps only. Lets call {\em  North-East} snakes the oriented {\em PDS} in $\mathcal S^1_{2D}$ with only north or east steps. {\em  North-West} snakes are defined similarly. The North-East snakes are counted by Fibonacci numbers, as proved in the next Proposition.
\begin{proposition} \label{propSNE}
The number of North-East snakes of length $n$, $s_{NE}(n)$, satisfies the relation
\begin{align}
s_{NE}(n)=s_{NE}(n-1)+s_{NE}(n-2) 
\label{sNE}
\end{align}
with $s_{NE}(0)=s_{NE}(1)=1$, which is equal to the $(n+1)^\text{e}$ Fibonacci number. Consequently the generating function of North-East snakes is
\begin{align}
S_{NE}(q)&=\sum_{n\geq 0}s_{NE}(n)q^n=\frac{1}{1-q-q^2}. 
\label{S_{NE}(z)}
\end{align}
\end{proposition}

\begin{proof} 
We use the well known interpretation of Fibonacci numbers as the number of tilings of a $b\times 1$ board with monominoes and dominoes (see~\cite{Benjamin2003proofs}). It is immediate that these tilings are in one-to-one correspondence with North-East snakes when we observe that starting from the left, we place dominoes in vertical positions and the following piece at the upper end of the domino, as illustrated in Figure~\ref{NonInsc_SNEE_bijection}. 
The monomino-domino tiling of a $n\times 1$ rectangle being counted by the $(n+1)^e$ Fibonacci number, North-East snakes must satisfy equation~\eqref{sNE} from which equation~\eqref{S_{NE}(z)} is obtained by  use of standard methods (see~\cite{Wilf1994}). 

\begin{figure}[h!]
\centering
\includegraphics[scale=0.8]{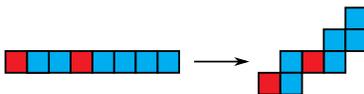}
\caption{Bijection from  monomino(red)-domino(blue) tilings of a row towards  North-East snakes
\label{NonInsc_SNEE_bijection}}
\end{figure}
\end{proof}

By symmetry there are as many North-East snakes as North-West snakes. However the non-oriented snakes consisting of only east steps are the same as the snakes with west steps only.  Also the unique snake of length~2 with one north step belongs to the two sets.  As a result of these observations, we obtain the following generating function of the set $\mathcal S^1_{2D}$:
\begin{align}
S^1_{2D}(q)=\frac{2}{1-q-q^2}-\frac{1}{1-q}-q^2.
 \label{S_{1N}(z)}
\end{align}
Consider $W\in\mathcal{S}^2_{2D}$. Identifying the last two consecutive north steps when moving from bottom to top, three disjoint parts appear in $W$ respectively illustrated in green, red and blue in Figure~\ref{NonInsc_S2N}: {\em i}) any   {\em PDS}, {\em ii}) two north steps and {\em iii}) either a North-East or a North-West snake. This combinatorial factorization is the key to obtain the generating function ${S}^2_{2D}(q)$.
\begin{figure}[h!]
\centering
\includegraphics[scale=0.8]{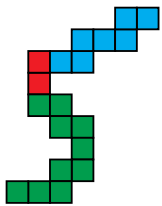}
\caption{A {\em PDS} in $\mathcal{S}^2_{2D}$ 
\label{NonInsc_S2N}}
\end{figure}
\begin{proposition} The generating function of {\it PDS} containing two consecutive north steps  at least once is
\begin{align}
S^2_{2D}(q)=\left(PDS_{2D}(q)-1+\frac{q^2}{1-q}\right)q^2\left(\frac{2}{1-q-q^2}-1\right). 
\label{S_{2N}(q)}
\end{align}
\end{proposition}
\begin{proof}
The three factors in equation~\eqref{S_{2N}(q)} correspond to each of the three parts of $W$. The left factor represents any  {\em PDS} where the subtraction of 1 guarantees that the {\em PDS} is not empty, i.e. that we add at least one cell in order to build the two last consecutive north steps.  Moreover since two consecutive north steps can be appended to a horizontal row of length at least $2$ in two different ways, the term $\frac{q^2}{1-q}$ is added. The middle factor $q^2$ in equation~\eqref{S_{2N}(q)} is coding the last two consecutive north steps (red cells in Figure~\ref{NonInsc_S2N}).  The right factor is twice the North-East snakes generating function~\eqref{S_{NE}(z)}, since there are as many North-East  snakes as North-West snakes, with $1$ subtracted since there is only one way to add an empty snake to the red north steps.
\end{proof}

Since $\mathcal{S}^1_{2D}$ and $\mathcal{S}^2_{2D}$ form a partition of the set $\mathcal{PDS}$, the sum of~\eqref{S_{1N}(z)} and~\eqref{S_{2N}(q)} gives  a functional equation for 2{\em D} {\em PDS}:
\begin{align*}
PDS_{2D}(q)=\left(PDS_{2D}(q)-1+\frac{q^2}{1-q}\right)q^2\left(\frac{2}{1-q-q^2}-1\right)+\frac{2}{1-q-q^2}-\frac{1}{1-q}-q^2
\end{align*}
which is immediately solved  to obtain the rational generating function
\begin{align}\label{PDS2D}
PDS_{2D}(q)&=\frac{q^5+q^3+q^2-2q+1}{(1-q)(1-2q-q^3)}\\
\nonumber
	&=1+q+2q^2+6q^3+14q^4+32q^5+72q^6+160q^7+ 354q^8+\cdots.
\end{align}
A formula giving the number of 2{\em D} {\em PDS} according to its length $n$ is obtained by converting the rational expression of $PDS_{2D}(q)$ into partial fractions:
\begin{align*}
pds_{2D}(n)&=\frac{1}{59}\sum_{\alpha:1-2\alpha-\alpha^3=0 } (13\alpha^{-1}+11-5\alpha)\alpha^{-n}-1\qquad n\geq 2,\\
pds_{2D}(0)&=pds_{2D}(1)=1.
\end{align*}

%%=============================================================
%% \subsection{Three dimensions} 
%%=============================================================
\subsection{Three dimensions} A $3${\em D} polyomino is a face-connected set of unit cubes in the discrete three dimensional space.  
A $3${\em D} {\em PDS} (towards the  $yz$ plane) is a $3${\em D} polyomino that starts at the origin and is formed with steps along the axis $x$, $y$ and $z$ where negative movement along the $y$ and $z$ axis are prohibited. The possible directions from one cell to a face-connected neighbour cell are denoted by $x^+,\ x^-,\ y^+$ and $z^+$.  The $3${\em D} {\em PDS}  are a natural extension of the $2${\em D} {\em PDS} which are snakes that can only have directions both positive and negative along the $x$ axis and positive along the $y$ axis.

$3${\em D} $\mathcal{ PDS}$ is partitionned into two sets, as shown in Figure~\ref{Parti3dPDS}:
\begin{itemize}
\item $3${\em D} {\em PDS} that do not contain two consecutive steps  in the set of directions $\{y^+,z^+\}$, $\mathcal{S}_{3D}^{1}$;
\item $3${\em D} {\em PDS} that contain two consecutive steps at least once taken in the set of directions $\{y^+,z^+\}$, $\mathcal{S}_{3D}^{2}$.
\end{itemize}
\begin{figure}[h!]
\subfigure[$\mathcal{S}_{3D}^{1}$]{{\includegraphics[scale=0.13]{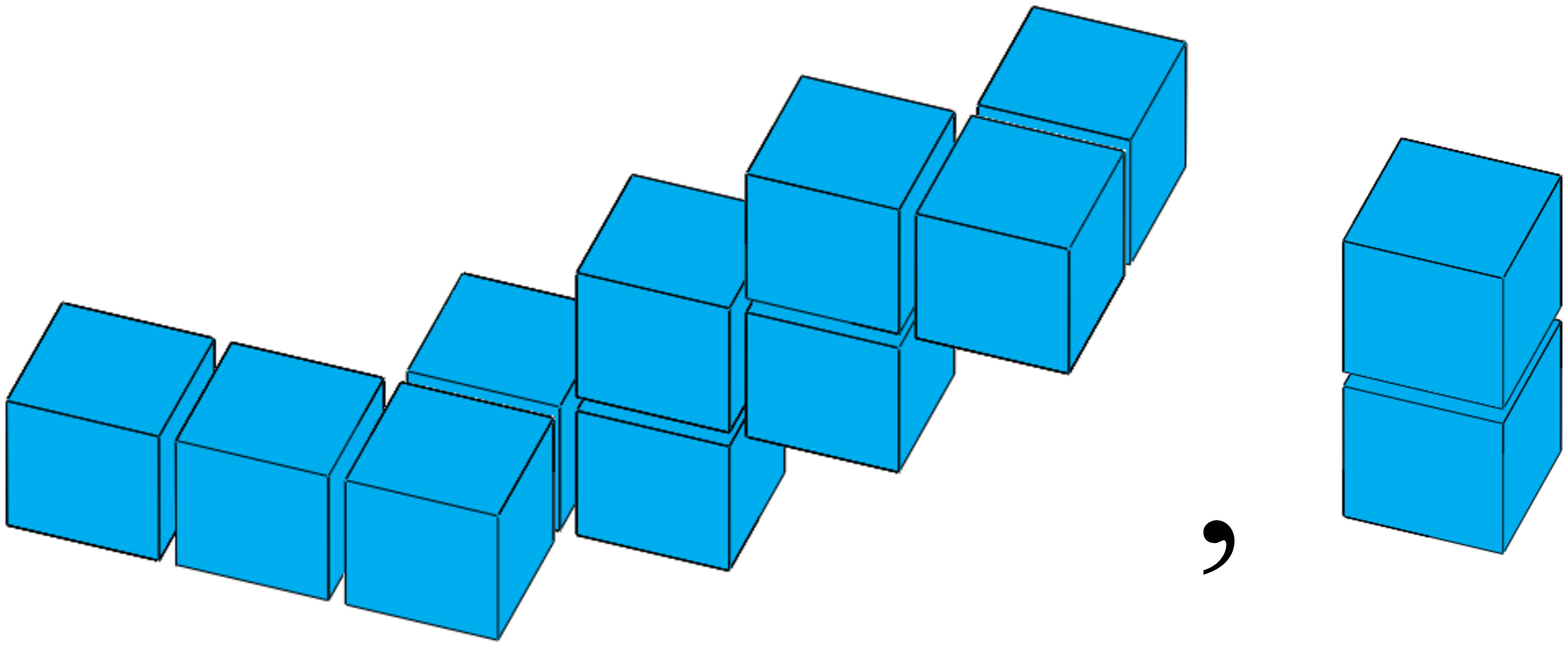}}}\qquad\qquad
\subfigure[$\mathcal{S}_{3D}^{2}$]{{\includegraphics[scale=0.18]{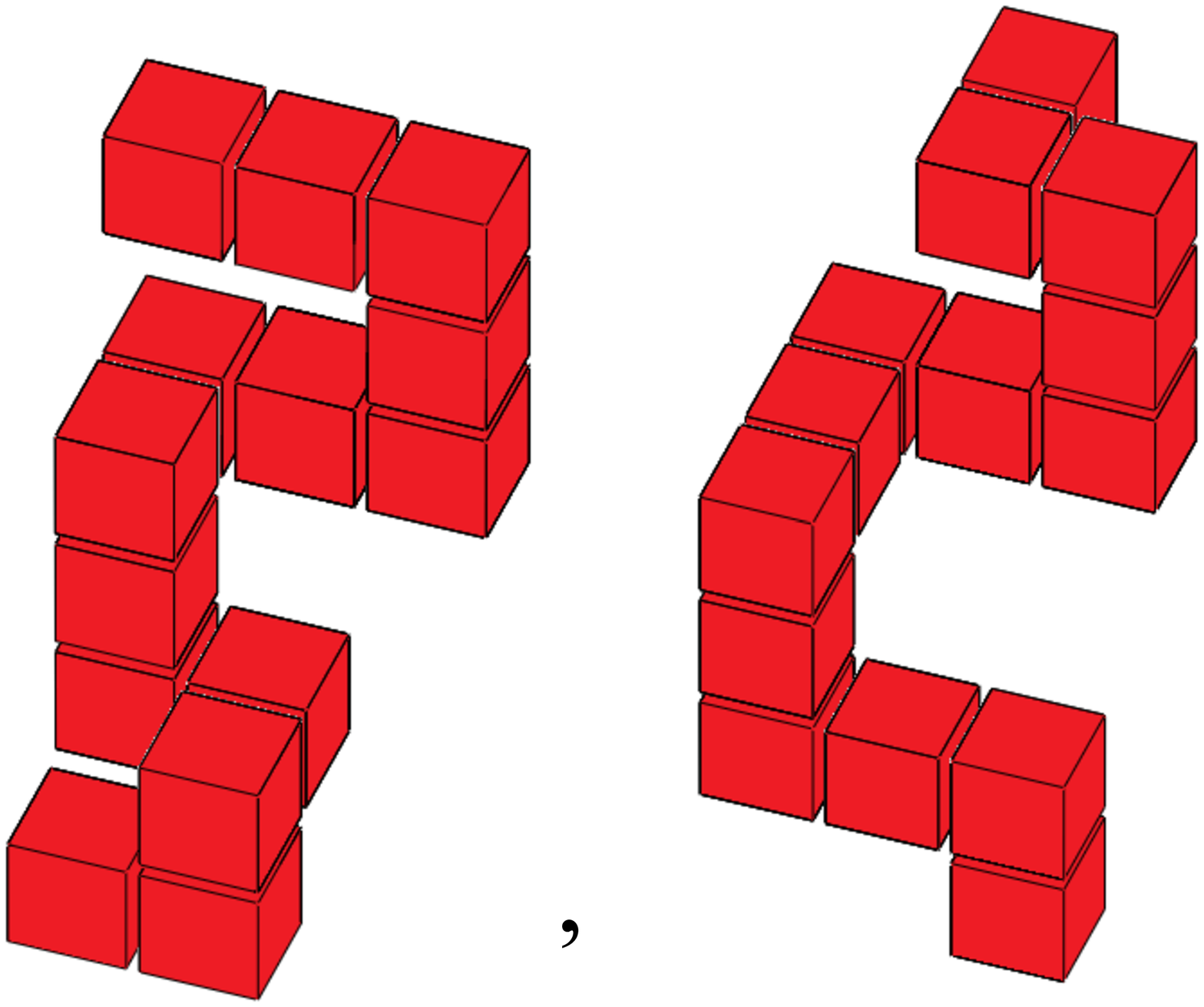}}}\qquad
\subfigure{{\includegraphics[scale=0.1]{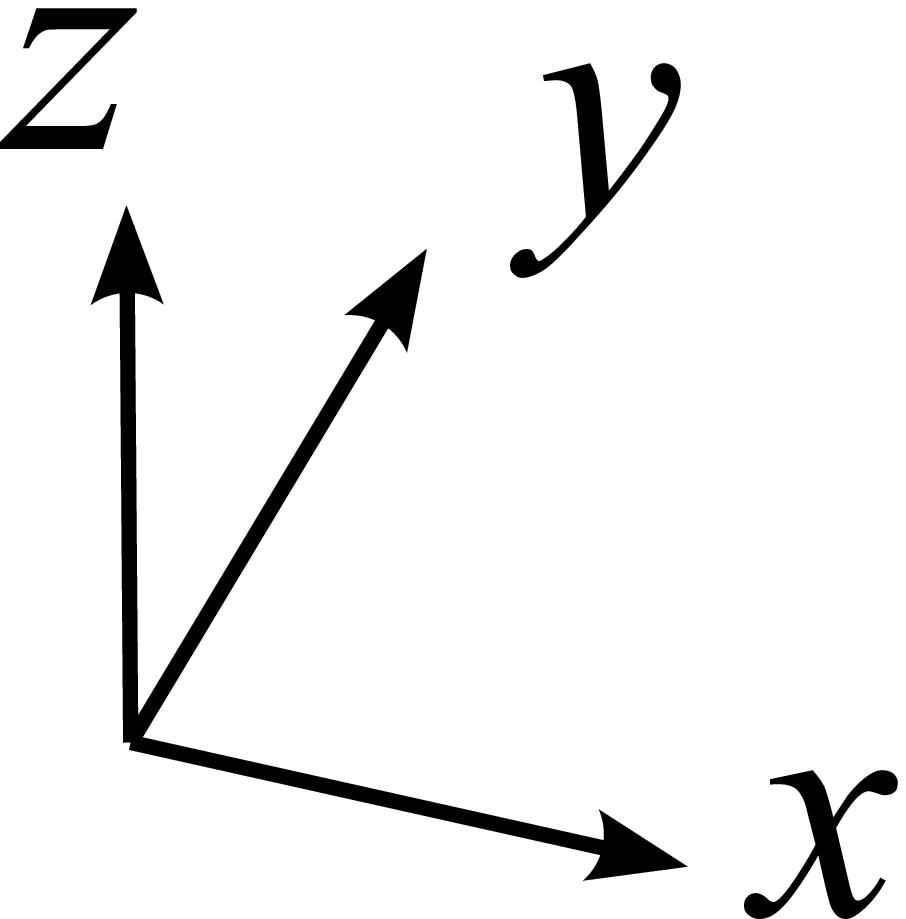}}}
\caption{A partition of the set of   $3${\em D} {\em PDS}\label{Parti3dPDS}}
\end{figure}
\begin{figure}[h!]
\centering
\subfigure[$x^+y^+z^+x^-$]{\makebox[2.2cm]{\includegraphics[scale=0.06]{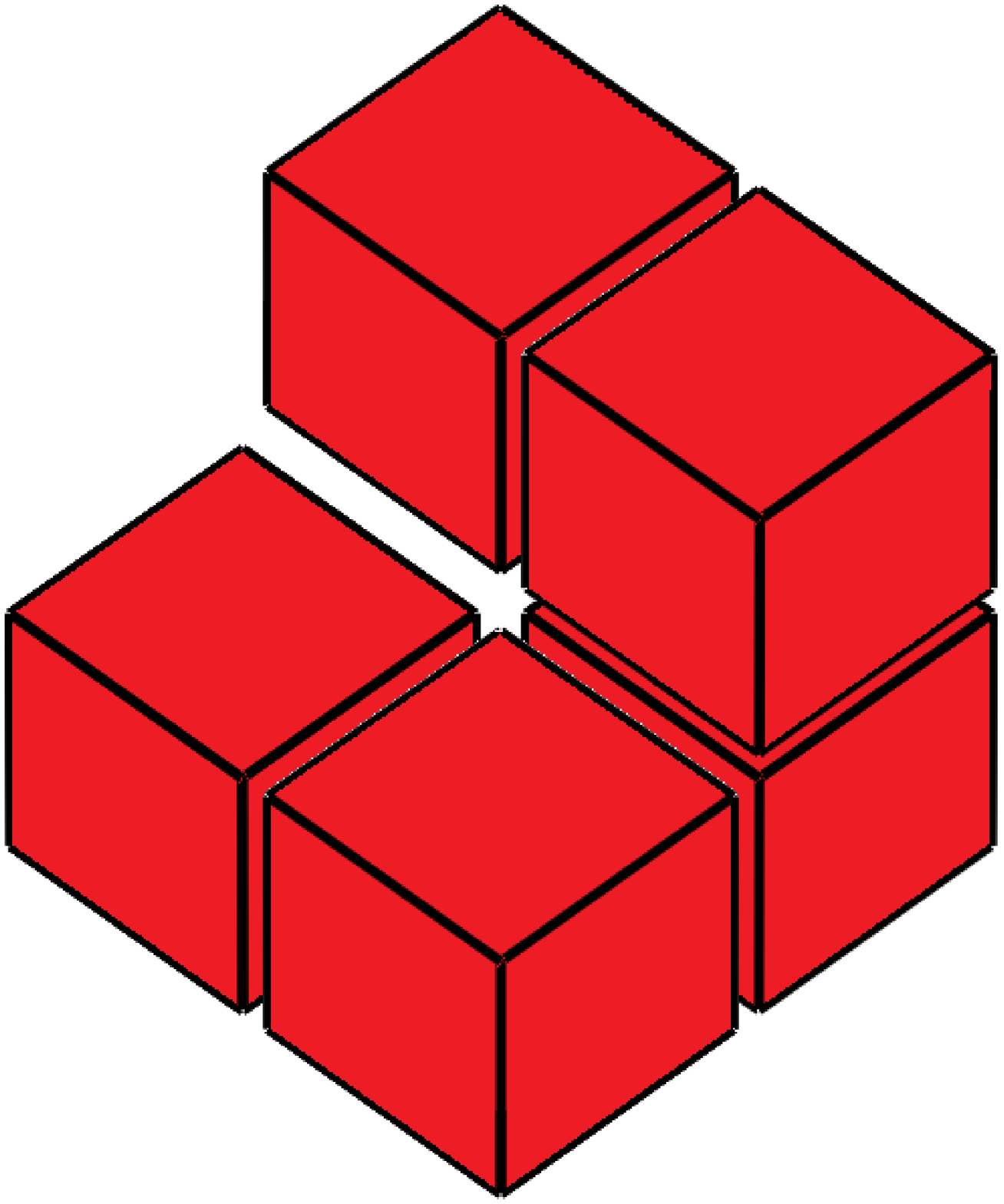}}}\qquad %{xyzx.png}
\subfigure[$x^-z^+z^+x^+$]{\makebox[2.2cm]{\includegraphics[scale=0.07]{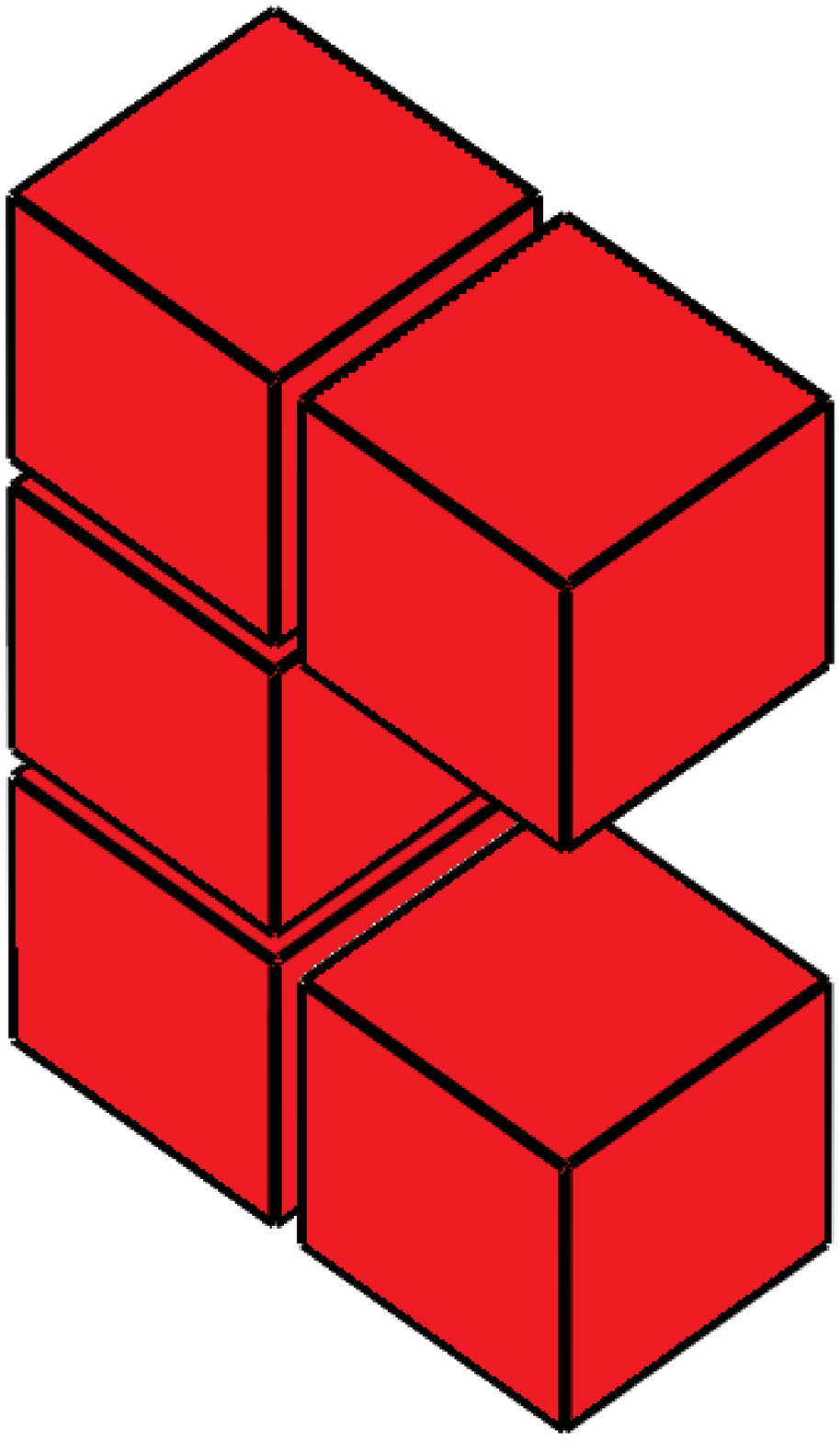}}}\qquad %{xzzx.png}
\subfigure[$x^-z^+y^+x^+$]{\makebox[2.2cm]{\includegraphics[scale=0.06]{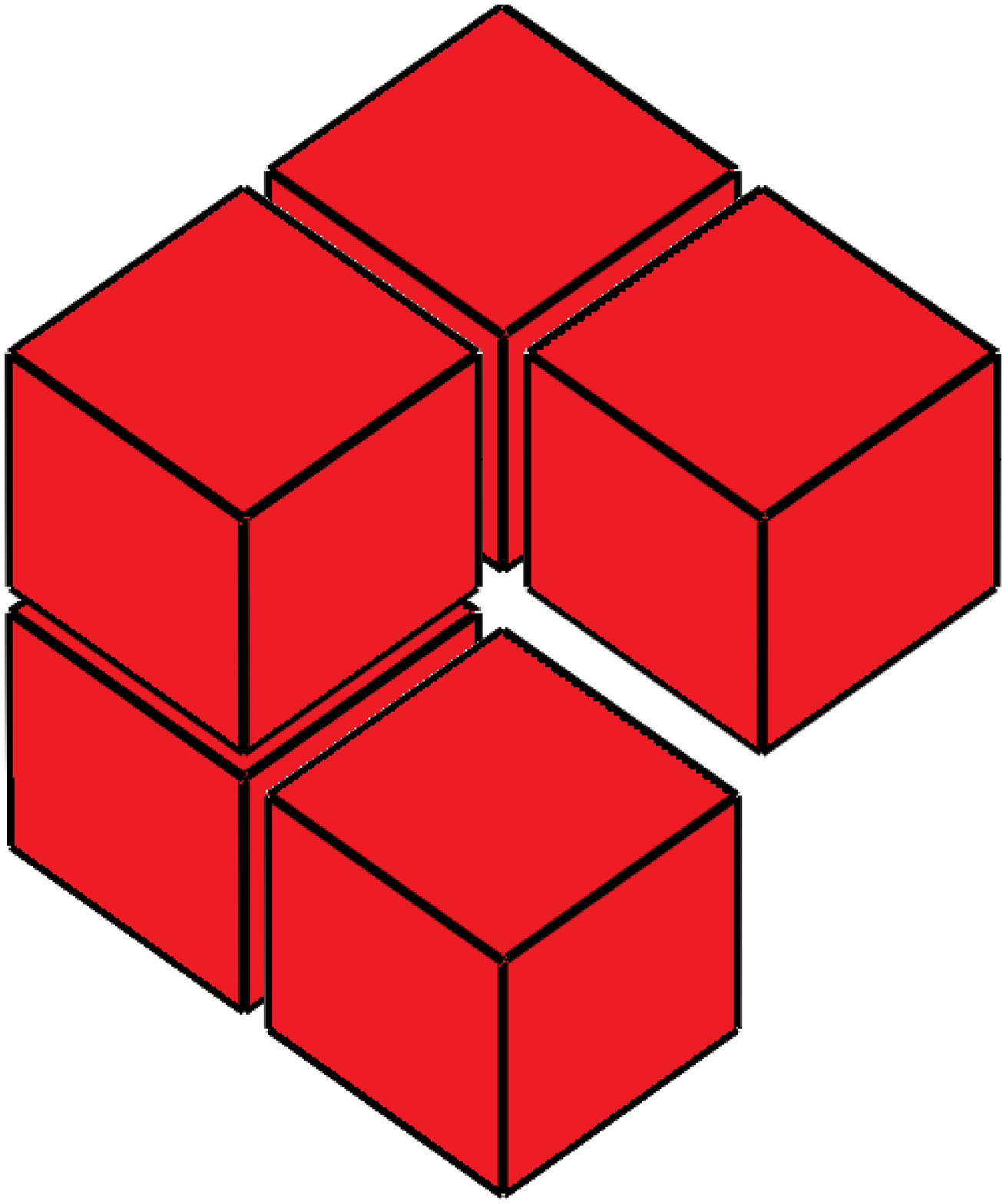}}}\qquad %{xzyx.png}
\subfigure[$x^+y^+y^+x^-$]{\makebox[2.2cm]{\includegraphics[scale=0.07]{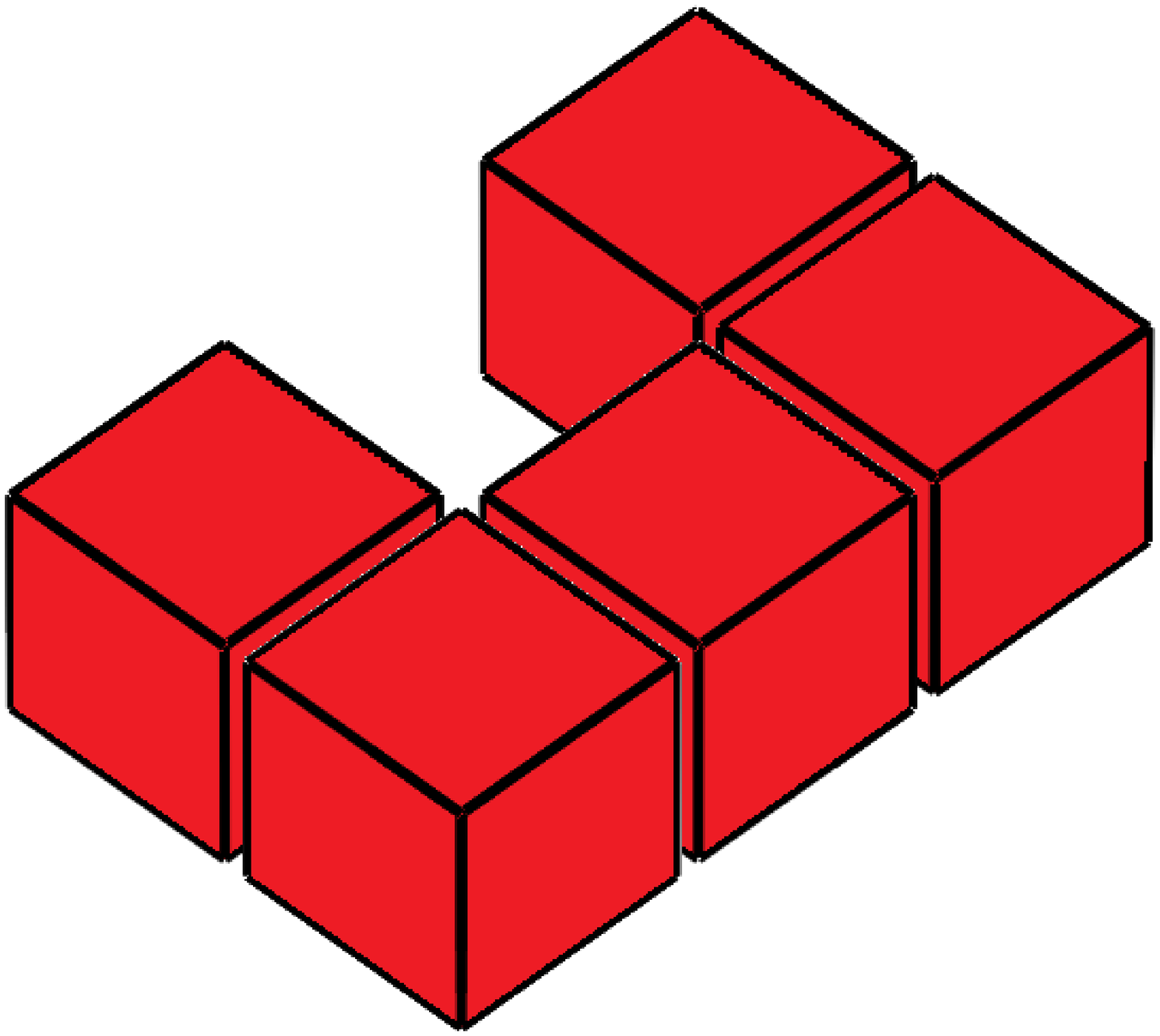}}}\qquad %{xyyx.png}
\subfigure{\includegraphics[scale=0.09]{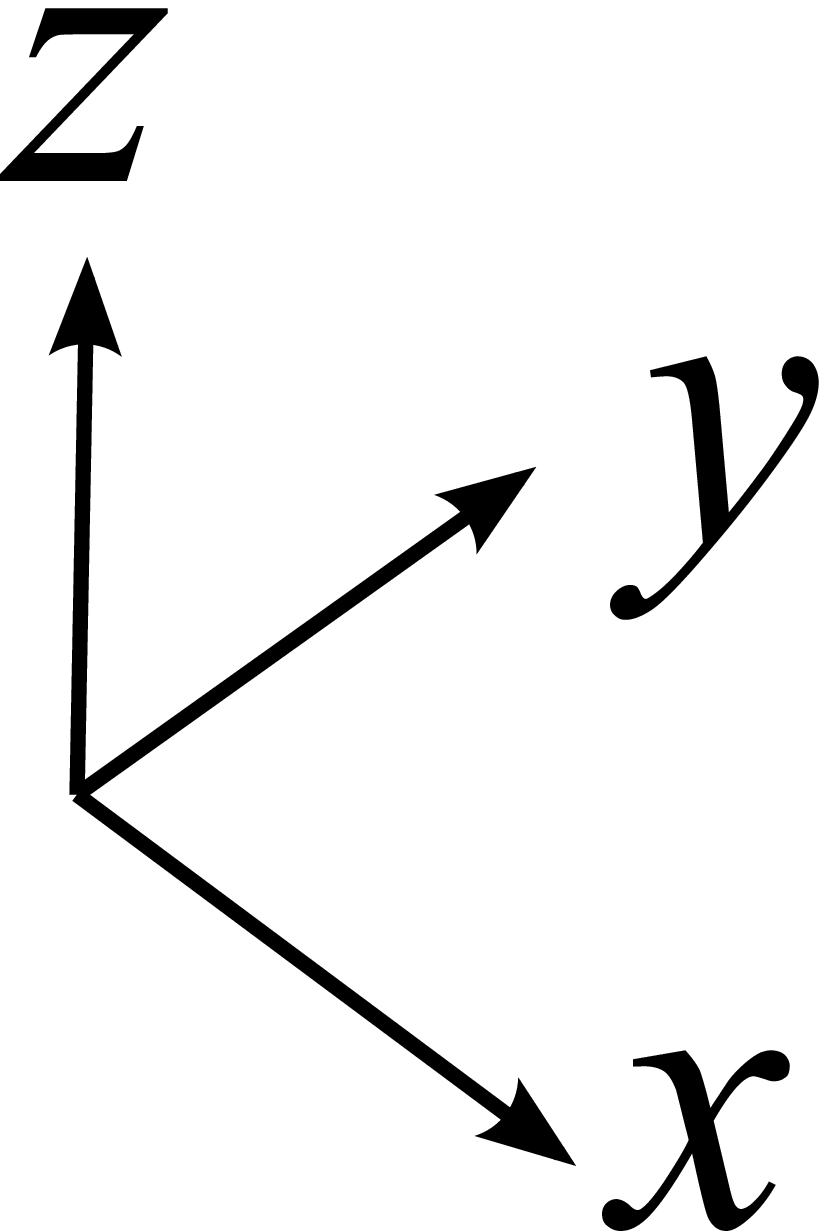}}
\caption{Four ways to move between $x^+$ and $x^-$ \label{Setyz}}
\end{figure}
Figure~\ref{Setyz} shows the four types of restricted directions in $\mathcal{S}_{3D}^{1}$.. 
The set $\mathcal{S}_{3D}^{1}$ is analogous to the set $\mathcal S^1_{2D}$, in which the snakes were not allowed to go from East to West, i.e from $x^+$ to $x^-$. As a consequence of prohibiting two consecutive steps in the set $\{y^+,z^+\}$, snakes belonging to $\mathcal{S}_{3D}^{1}$ cannot contain both directions $x^+$ and $x^-$.  So snakes in $\mathcal{S}_{3D}^{1}$ either have $x^+$ or $x^-$ steps but not both, leading to describe $PDS_{x^+}$ and $PDS_{x^-}$ snakes. Lets denote $\mathcal{PDS}_{x^+}$ the set of snakes in  $\mathcal{S}_{3D}^{1}$ with $x^+$ moves.  $\mathcal{PDS}_{x^-}$ is similarly defined.
\begin{proposition}
The number of $PDS_{x^+}$ of length $n$, denoted $pds_{x^+}(n)$, satisfies the relation
\begin{align}
pds_{x^+}(n)=pds_{x^+}(n-1)+2pds_{x^+}(n-2), \label{pds_{x^+}(n)}
\end{align}
with $pds_{x^+}(0)=pds_{x^+}(1)=1$ and the length generating function of $PDS_{x^+}$ is
\begin{align}
PDS_{x^+}(q)=\sum_{n\geq 0}pds_{x^+}(n)q^n=\frac{1}{(1+q)(1-2q)} \label{PDS_{x^+}(z)}
\end{align}
so that
\begin{align}
pds_{x^+}(n)=\frac{1}{3}\left(2^{n+1}+(-1)^n\right). \label{pds_{}x^+(n)}
\end{align}
\end{proposition}
\begin{proof}
A $PDS_{x^+}$ of length $n$ starts with a step in $\{x^+, y^+,z^+\}$. If the first step is $x^+$, then the second step is $x^+$, $y^+$ or $z^+$. Consequently the snake is completed by any $PDS_{x^+}$ of length $n-1$. If the first step is $y^+$ or $z^+$, then it must be followed by an $x^+$ step and, to complete the snake, any $PDS_{x^+}$ of length $n-2$ can be concatenated. This yields equation~\eqref{pds_{x^+}(n)} from which we deduce equation~\eqref{PDS_{x^+}(z)}. The exact formula~\eqref{pds_{}x^+(n)} is obtained from equation~\eqref{PDS_{x^+}(z)}.
\end{proof}

The  sets $\mathcal{PDS}_{x^+}$ and $\mathcal{PDS}_{x^-}$ are not disjoint. The horizontal rows of cells (formed with  $x^+$ steps only or with  $x^-$ steps only), the unique snake with one $y^+$ and the unique snake with one $z^+$ belong to the two sets. As a consequence the generating function of $3${\em D} {\em PDS} that does not contain two consecutive steps  in the set $\{y^+,z^+\}$ is
\begin{align}
S_{3D}^{1}(q)=\frac{2}{(1+q)(1-2q)}-\frac{1}{1-q}-2q^2. \label{S_{3D}^{1yz}(z)}
\end{align}
A snake $V$ in $\mathcal{S}_{3D}^{2}$ can always be decomposed in three distinct parts  shown respectively in red, green and blue in Figure~\ref{3partsS2yz}:  {\em i}) any $3${\em D} {\em PDS},  {\em ii}) the last two consecutive directions taken in the set $\{y^+,z^+\}$ when going from bottom to top  in $V$,  {\em iii}) either a $PDS_{x^+}$ snake or a $PDS_{x^-}$ snake.
\begin{figure}[h!]
\centering 
\subfigure[$z^+z^+$]{\includegraphics[scale=0.19]{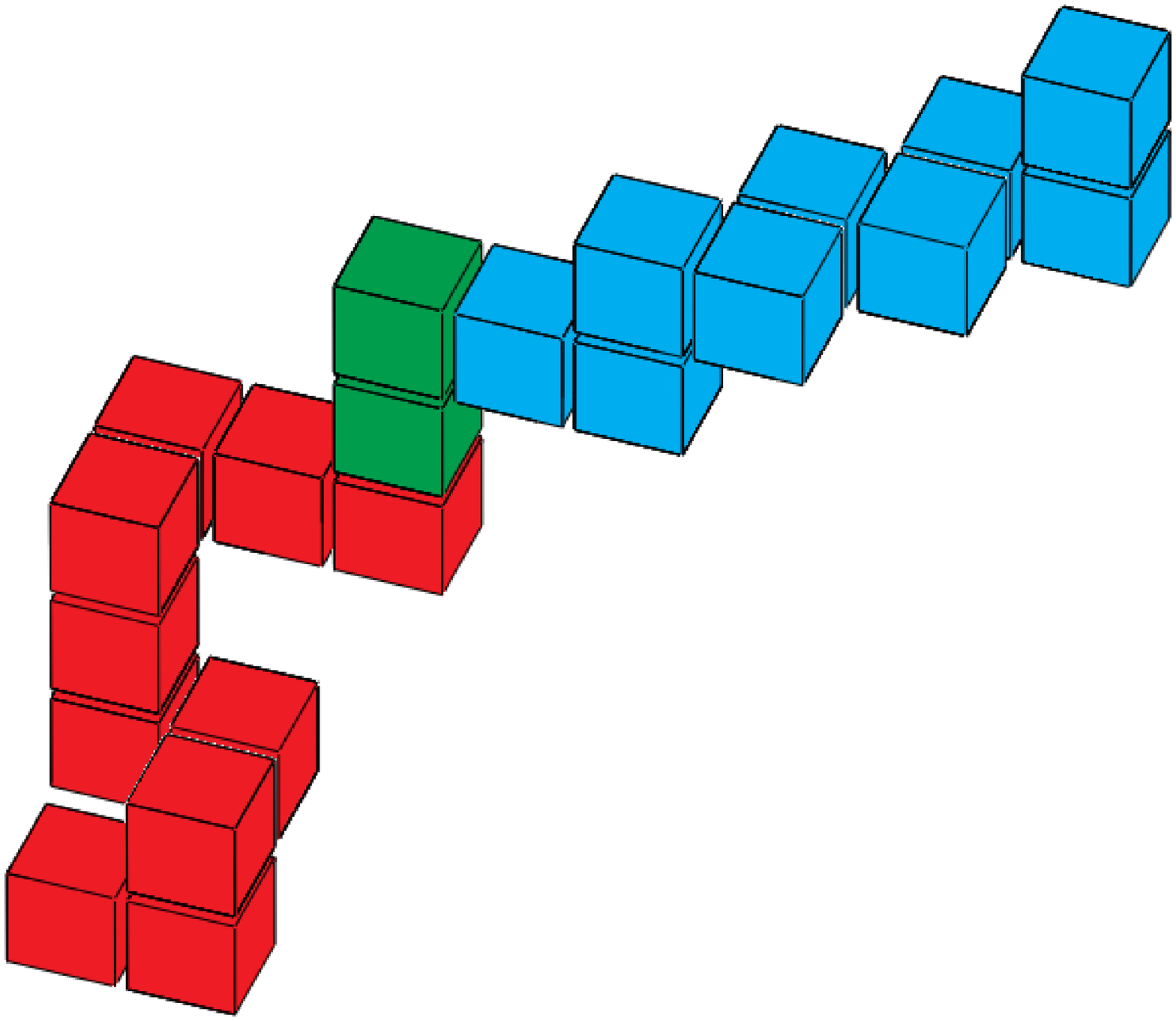}}\qquad\qquad 
\subfigure[$z^+y^+$]{\includegraphics[scale=0.2]{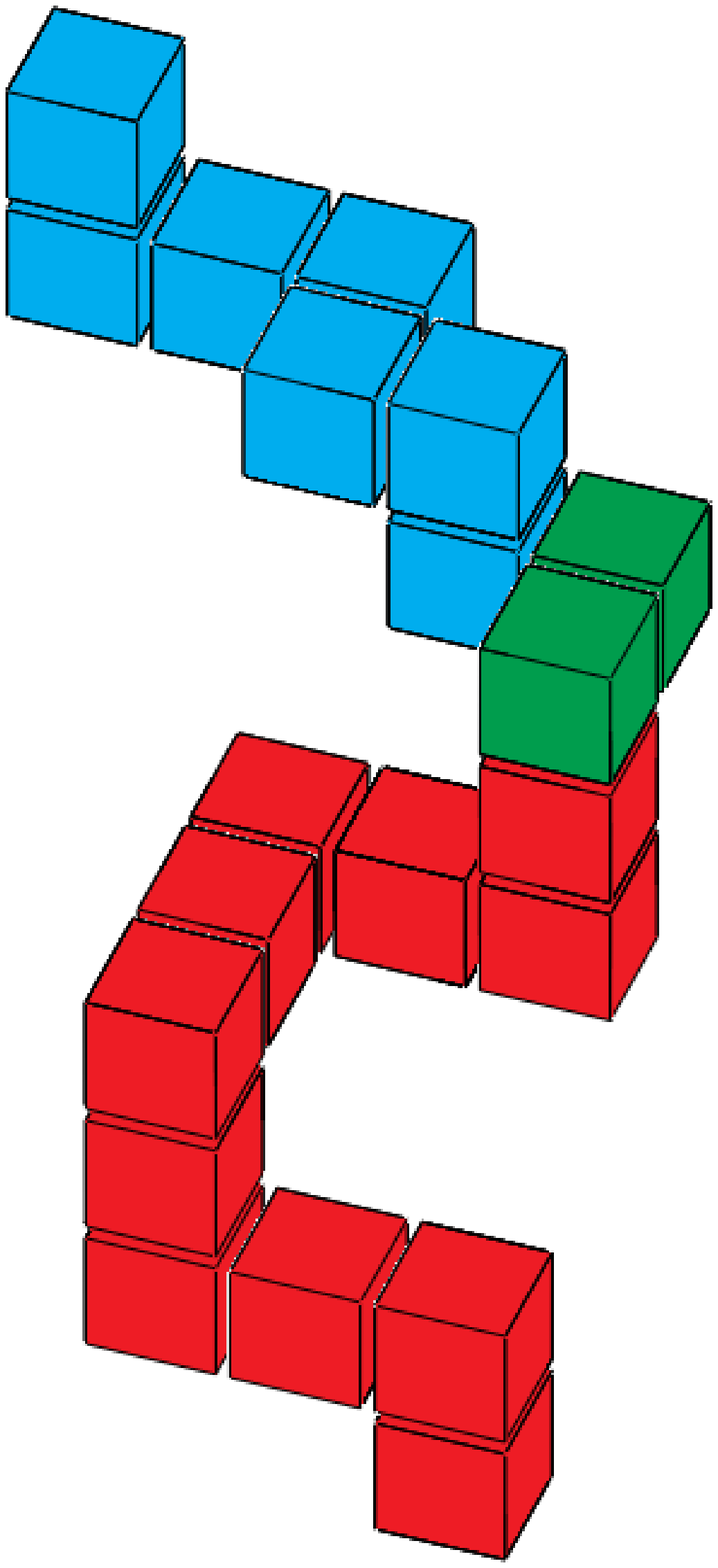}}\qquad\qquad 
\subfigure{{\includegraphics[scale=0.1]{Referentielle2.eps}}}
\caption{Three parts of $\mathcal S_{3D}^{2}$\label{3partsS2yz}}
\end{figure}
\begin{proposition}
The generating function of the set $\mathcal{S}_{3D}^{2}$  is
\begin{align}
{S}_{3D}^{2}(q)=\left( PDS_{3D}(q)-1+\frac{q^2}{1-q}\right) 4q^2 \left(\frac{2}{(1+q)(1-2q)}-1\right)\label{{S}_{3D}^{2yz}}.
\end{align}
\end{proposition}
\begin{proof}
The three parts described earlier yield the three factors of equation~\eqref{{S}_{3D}^{2yz}},  in similarity with the $2${\em D} case. The left factor characterizes any $3${\em D} {\em PDS} where 1 is subtracted and $\frac{q^2}{1-q}$ is added respectively to assure that the $3${\em D} {\em PDS} is not empty and because there are two ways to append a given choice of two consecutive steps to a horizontal row. The middle factor $4q^2$ represents the four possibilities of two consecutive steps taken in  $\{y^+,z^+\}$. The right factor is twice the generating function $PDS_{x^+}(q)$ of equation  ~\eqref{PDS_{x^+}(z)}, since  $PDS_{x^+}$ and $PDS_{x^-}$ are symmetric snakes. Moreover 1 is subtracted to take into account that an empty snake can be appended in only one way to  two consecutive steps in  $\{y^+,z^+\}$.
\end{proof}

The one-variable generating function of $3${\em D}  {\em PDS}, denoted $PDS_{3D}(q)$, is defined by
$PDS_{3D}(q)=\sum_{n\geq 0} pds_{3D}(n) q^n$ where $pds_{3D}(n)$ is the number of $3${\em D} {\em PDS} of length $n$.
Because the sets $\mathcal{S}_{3D}^{1}$ and $\mathcal{S}_{3D}^{2}$ form a partition of  $3${\em D} ${\mathcal{PDS}}$, the sum of equations~\eqref{S_{3D}^{1yz}(z)} and~\eqref{{S}_{3D}^{2yz}} gives a functional equation for $3${\em D} {\em{PDS}}:
\begin{align*}
PDS_{3D}(q)&=\left( PDS_{3D}(q)-1+\frac{q^2}{1-q}\right) 4q^2 \left(\frac{2}{(1+q)(1-2q)}-1\right)\\ &\quad+\frac{2}{(1+q)(1-2q)}-\frac{1}{1-q}-2q^2.
\end{align*}
This equation is directly solved to obtain
\begin{align}\label{PDS3D}
PDS_{3D}(q)&=\frac{4q^5+2q^4+2q^2-3q+1}{(1-q)(1-3q-4q^3)}\\
\nonumber
	         &= 1+q+3q^2+13q^3+45q^4+153q^5+517q^6+1737q^7+5829q^8+\cdots.
\end{align}
Expressing the rational generating series $PDS_{3D}(q)$ into partial fractions, the number of $3${\em D} {\em PDS} of length $n$ is given by
\begin{align*}
pds_{3D}(n)&=\frac{1}{12}\sum_{\alpha:1-3\alpha-4\alpha^3=0} (\alpha^{-1}+1)\alpha^{-n}-1 \qquad n\geq 2,\\
pds_{3D}(0)&=pds_{3D}(1)=1.
\end{align*}

%%%%%%%%%%%%%%%%%%%%%%%%%%%%%%%%%%%%%%%%%%%
%\subsection{N dimensions}
%%%%%%%%%%%%%%%%%%%%%%%%%%%%%%%%%%%%%%%%%%

\subsection{N dimensions $\mathbf{PDS}$} 
There is a natural  extension of partially directed snakes  in two and three dimensions  to the $N$-dimensional lattice $\mathbb{Z}^N$.These $N$-dimensional  snakes are denoted $ND$ $PDS$ and their set of possible directions along the coordinate axes  is denoted $\{x_1^+,x_1^- ,x_2^+,\ldots , x_N^+\}$. Steps in the negative direction is allowed only along the $x_1$ axis and forbidden along the other axis $x_2,x_3,\ldots , x_n$. 
We partition again the set $\mathcal{ND\; PDS}$ in two  disjoint subsets:
\begin{itemize}
\item $ND$ {\em PDS} that do not contain two consecutive steps  in the set  $\{x_2^+,x_3^+,\ldots , x_N^+\}$: $\mathcal{S}_{ND}^{1}$;
\item $ND$ {\em PDS} that contain two consecutive steps  in the set $\{x_2^+,x_3^+,\ldots , x_N^+\}$ at least once:  
$\mathcal{S}_{ND}^{2}$.
\end{itemize}
If we denote by $\mathcal{PDS}_{x_1^+}$ the set of snakes in $\mathcal{S}_{ND}^{1}$ with $x_1^+$ steps only, then the cardinalty of that set satisfies 
\begin{align*}
pds_{x_1^+}(n)&=pds_{x_1^+}(n-1)+(N-1)pds_{x_1^+}(n-2)
\end{align*}
with $pds_{x_1^+}(0)=pds_{x_1^+}(1)=1$ from which we deduce the length generating function
\begin{align*}
PDS_{x_1^+}(q)&=\frac{1}{1-q-(N-1)q^2}
\end{align*}
then,  using  arguments  similar to the $2D$ and $3D$ cases, we deduce the exact formula 
\begin{align*}
pds_{x_1^{+}}(n) = \frac{(2 - 2N)\left(\frac{2 - 2N}{1 - \sqrt{4N - 3}}\right)^{n}}{ (1 - \sqrt{4N-3})\sqrt{4N - 3}} + \frac{(2N - 2) \left(\frac{2 - 2N}{1 + \sqrt{4N - 3}} \right)^n}{(1 + \sqrt{4N - 3})\sqrt{4N - 3}}
\end{align*}
and the following rational expressions for the generating functions of the two sets $\mathcal{S}_{ND}^{1}$ and $\mathcal{S}_{ND}^{2}$:
\begin{align*}  
& S_{ND}^{1}(q) = \frac{2}{q^{2}(1-N) - q + 1} - \frac{1}{1 - q} - (N-1)q^{2},\\
& S_{ND}^{2}(q) = \left(PDS_{ND}(q) - 1 + \frac{q^{2}}{1 - q} \right) (N-1)^{2}q^{2} \left(\frac{2}{q^{2}(1-N) - q + 1} - 1 \right).
\end{align*}
Then solving the functional equation 
\begin{align*} 
 PDS_{ND}(q) = S_{ND}^{1}(q) + S_{ND}^{2}(q),
\end{align*}
we obtain the following rational expression for the generating function of $N$ dimensional partially directed snakes
\begin{align} \label{PDSND}
 PDS_{ND}(q) = \frac{(N-1)^{2}\ q^{5} + (N-1)(N-2)q^{4} - (N-1)(N-3)q^{3} + (N-1)q^{2} - Nq + 1}{(1-q)(1 - Nq - (N-1)^{2}\ q^{3})}.
\end{align}
Equations (\ref{PDS2D})  and (\ref{PDS3D})  are the special cases $N=2$ and $N=3$ of equation  (\ref{PDSND}).

%=============================================================
% \section{Inscribed partially directed snakes} 
%=============================================================

\section{Inscribed partially directed snakes} \label{SecInscribedPDS}
\label{insc}
Let $pds(b,k,n)$ be the number of partially directed snakes of length $n$ inscribed in a rectangle of size $b\times k$ and define the corresponding generating function $PDS(s,t,q)$ as 
\begin{align*}
PDS(s,t,q)=\sum_{b,k,n}pds(b,k,n)s^bt^kq^n.
\end{align*}

The goal of this section is to present  rational expressions for the generating functions $PDS_b(t,q)$ of partially directed snakes of width $b$. In~\cite{Bm} functional equations were first established and then solved to obtain generating functions for bargraph polyominoes. The method used in the present section is different and consists in starting from a primitive snake for which we know the generating function,  stacking on top of each other fundamental pieces of snakes that will produce all possible snakes of a given width $b$ and then modify accordingly the corresponding generating function. Thus the constructions of the snakes and of the generating functions evolves simultaneously. The rules for establishing the correspondence are quite straightforward: $i)$ Stacking one family of snakes on top of another corresponds to the product of the corresponding generating functions. $ii)$ Picking in a bag an arbitrary number of snakes belonging to a family $\mathcal{A}$ and then stacking these snakes on top 
of the current snake 
is represented by  the multiplication of the current generating function with $\frac{1}{1-A(s)}$ where $A(s)$ is the generating function of the family $\mathcal{A}$. 

Let $b$ and $k$ be respectively the horizontal and vertical lengths of  the rectangle. We  need four kinds of $PDS$ to construct our generating functions: crossings, bubbles, endings and pillars. The snakes in the last family  are used to connect snakes  of the other kinds and are defined in section~\ref{Sec4.1}.  These four kind of snakes are combined in different ways to form four disjoint families of snakes whose union will be the entire set $\mathcal{PDS}(b,k,n)$. In section~\ref{Sec4.2} generating functions of these snakes are presented.

%%%%%%%%%%%%%%%%%%%%%%%%%%%%%%%%%%%%%%%%%%%%%%%%%%%%
% \subsection{Families of inscribed partially directed snakes} 
%%%%%%%%%%%%%%%%%%%%%%%%%%%%%%%%%%%%%%%%%%%%%%%%%%%%
\subsection{Families of inscribed partially directed snakes}\label{Sec4.1} A {\it bubble} is a snake with its tail and head in the same column adjacent to a vertical side of the rectangle and its other cells in the other columns in the rectangle  that do not touch the other side of the rectangle so that a bubble is not an inscribed snake. The set of bubble snakes with head and tail on a given side of the rectangle, say left,  that occupies $r$ columns ($r<b$)  is  denoted $\mathcal{B}_r(k,n)$ and  $r$ is called the width of the bubble. 
A {\it crossing}  of width $b$ is an inscribed {\it PDS }whose only two cells adjacent to the left and right sides of the rectangle are the tail and head of the snake.  When the tail touches the left side, the head touches the right side and vice versa. The set of crossings of length $n$ inscribed in a $b\times k$ rectangle with bottom cell in its left corner is denoted $\mathcal{C}(b,k,n)$. 
An {\it ending} is a partially directed snake whose only cell in the column on one side of the rectangle is the tail. The set of endings with tail on a given side of the rectangle is denoted $\mathcal{E}(b,k,n)$. Vertical rows of cells on a column adjacent to a side of the rectangle are called {\it pillars} with length generating function noted $P(t,q)$.
Orientation is not relevant in these four kind of  snakes which are therefore non-oriented snakes. Figure~\ref{famSnake} shows a sample of snakes from each kind where the blue cells in Figures~\ref{fambubble} and~\ref{famEnd} are part of a pillar necessary when they are piled over an existing snake.
\begin{figure}[htbp]
\begin{center}
\centering
\subfigure[Crossing]{\makebox[2.7cm]{\includegraphics[scale=0.8]{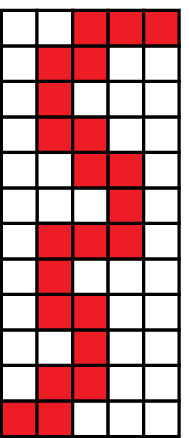}}}
\subfigure[Bubble with pillar]{\makebox[3.2cm]{\includegraphics[scale=0.8]{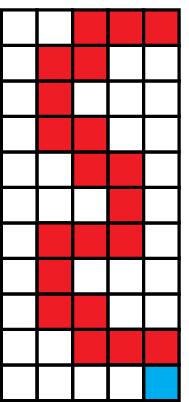}\label{fambubble}}}\quad
\subfigure[Ending with pillar]{\makebox[3.2cm]{\includegraphics[scale=0.8]{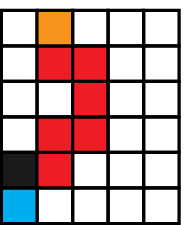}\label{famEnd}}}
\subfigure[Pillar]{\makebox[2.7cm]{\includegraphics[scale=0.8]{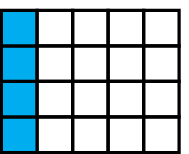}\label{fampillar}}}
\caption{Four kinds of snakes}
\label{famSnake}
\end{center}
\end{figure}

Using a symmetry argument, we can stack on top of each other snakes of each of the families just defined and shown in Figure~\ref{fam}.  In this stacking process, pillars are used as connectors between a snake in one of the three other families  and the current snake. Bubble  and crossing snakes may  appear an arbitrary number of times in a snake but endings can be used only once at each end of the snake to terminate the stacking process. This observation allows the use of ordinary product of generating functions to describe sets of inscribed $\mathcal{PDS}$. In fact we will show that all $PDS$ are obtained in this way and that they belong to exactly one of the following four families of inscribed snakes.

Family $\mathcal{F}_1$ is made of two crossings between which there is an arbitrary number of crossings and bubbles. Endings can be added at each end of the two mandatory crossings. Here we use crossings and bubbles up to vertical symmetry and the factor $2$ in equation (\ref{f1}) refers to the two choices offered for the lowest crossing in the snake : left or right side of the rectangle for the bottom cell.

Family $\mathcal{F}_2$ has one mandatory crossing and one mandatory bubble. The stacking process starts with a crossing and finishes with a bubble. Between these two snakes,  there is an arbitrary number of crossings and bubbles.  Again ending snakes can be added to each extremity of this structure.  Here we use a horizontal and a vertical symmetry  to collect all the snakes of this family which explains the factor $4$ in equation (\ref{f2}). 

Family $\mathcal{F}_3$ has one crossing snake between two mandatory bubbles.  On one side of the crossing, there is at least one bubble with no crossing. On the other side of the crossing, there are an arbitrary number of crossings and bubbles terminating with one bubble. Each extremity of this structure has a possible ending. We use a horizontal symmetry to collect all snakes of this family which is why the factor $2$ appears in equation (\ref{f3}). 

Finally the set $\mathcal{F}_4$ is made of a unique crossing, no bubble, with possible ending snakes on each extremity. We multiply by two the expression in equation (\ref{f4}) because of a vertical symmetry. We remove the snake $H$ which consists of a single horizontal row from this set for it would otherwise be counted twice. 

Samples from  families $\mathcal{F}_1$ to $\mathcal{F}_4$ are shown in Figure \ref{fam} and their structure is explicitly described in the following equations where we use the notation $()^*$ to mean that the family of snakes inside the parenthesis can be appended an arbitrary number of times.
\begin{flalign}
	&\mathcal{F}_{1} := 2(\varnothing \, \cup [(\varnothing \, \cup E) \times P]) \times C \times (P \times (C \cup B))^{*} \times P \times C \times (\varnothing \, \cup [P \times (\varnothing \, \cup E)])\label{f1}\\
	&\mathcal{F}_{2} := 4(\varnothing \, \cup [(\varnothing \, \cup E) \times P]) \times C \times (P \times (C \cup B))^{*} \times P \times B \times (\varnothing \, \cup [P \times (\varnothing \, \cup E)])\label{f2}\\
	&\mathcal{F}_{3} := 2(\varnothing \, \cup [(\varnothing \, \cup E) \times P]) \times B \times (P \times B)^{*} \times P \times C \times (P \times (C \cup B))^{*} \times P \times B \times \nonumber\\
	&\phantom{aaaaaa}(\varnothing \, \cup [P \times (\varnothing \, \cup E)])\label{f3}\\
	&\mathcal{F}_{4} := 2(\varnothing \, \cup [(\varnothing \, \cup E) \times P]) \times C \times (\varnothing \, \cup [P \times (\varnothing \, \cup E)]) - H\label{f4}
\end{flalign}
These equations are transformed into generating functions in section~\ref{genfunc}. As explained above, the factors $2$ and $4$ in equations (\ref{f1}) to (\ref{f4})  are the  result of the use of symmetries to obtain all snakes in the corresponding families.  

\begin{figure}[htbp]
\centering
	\subfigure[$\mathcal{F}_{1}$]{\hspace{0.8cm}\includegraphics[scale=0.8]{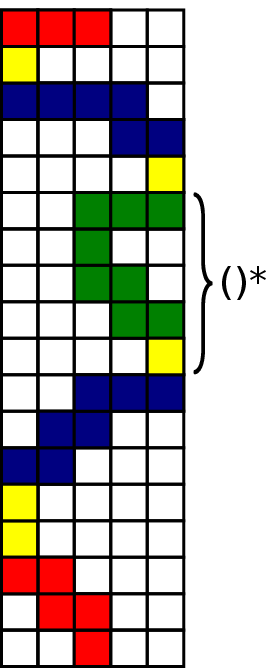}}
	\subfigure[$\mathcal{F}_{2}$]{\hspace{0.8cm}\includegraphics[scale=0.8]{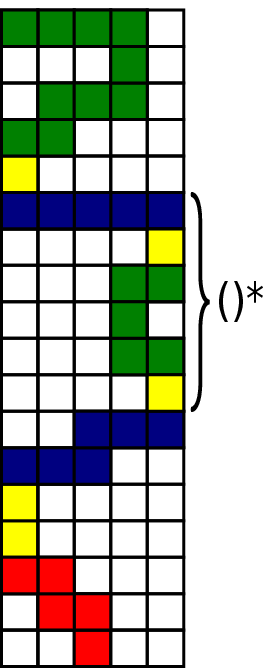}}
	\subfigure[$\mathcal{F}_{3}$]{\hspace{0.8cm}\includegraphics[scale=0.8]{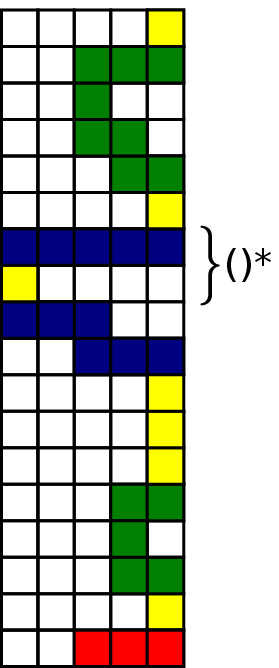}}\hspace{0.8cm}
	\subfigure[$\mathcal{F}_{4}$]{\includegraphics[scale=0.8]{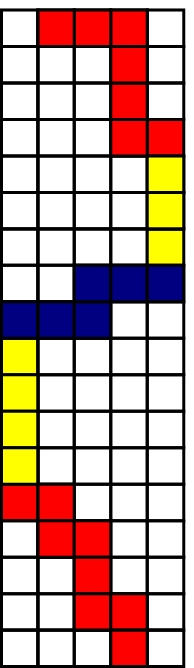}}
	\caption{The four families of {\em PDS} where a crossing is blue, a bubble is green, an ending is red, a pilar is yellow and the notation $()^*$ means that the family of snakes inside the parenthesis can be appended an arbitrary number of times}
	\label{fam}
\end{figure}
By construction, families $\mathcal{F}_1$ to $\mathcal{F}_4$ are disjoint and any inscribed $PDS$  belongs to one of these families. We thus obtain the following proposition  
\begin{proposition}\label{proppart}
For all positive integers $b,k$ and $n\geq b+k-1$, we have 
\begin{align*}
{\mathcal{PDS}}(b,k,n)=\mathcal{F}_1(b,k,n)+\mathcal{F}_2(b,k,n)+\mathcal{F}_3(b,k,n)+\mathcal{F}_4(b,k,n).
\end{align*}
\end{proposition}
\begin{proof}
The proof is based on the following observations. {\em i}) Each inscribed $PDS$ contains at least one crossing. {\em ii})  Any inscribed snake can be factored in a unique way as a juxtaposition of crossings, bubbles, endings and pillars. {\em iii}) Appending an ending snake at one step terminates the stacking process  in one direction. The head of the ending snake then becomes the head or tail of the entire snake. Any snake is the juxtaposition of a central part with possible endings appended on each side. If the central part begins and ends with two different crossings then we have the set $\mathcal{F}_1$. If the central part begins with a crossing and ends with a bubble then we have the set $\mathcal{F}_2$. If the central part
begins and ends with a bubble then we have $\mathcal{F}_3$. Finally if the center has exactly one crossing and no bubble then we have $\mathcal{F}_4$. There is no other possibility for the central part. \\
\end{proof}

%%%%%%%%%%%%%%%%%%%%%%%%%%%%%%%%%%%%%%%%%%%%%%%%%%%%
% \subsection{Generating Functions} 
%%%%%%%%%%%%%%%%%%%%%%%%%%%%%%%%%%%%%%%%%%%%%%%%%%%%
 \subsection{Generating Functions} \label{Sec4.2}
 \label{genfunc} Recall that the variables $s,\ t,\ q$ provide respectively the width and height   of the rectangle and the length of the snake. Since the column in which a pillar occur is accounted for in the other families of snakes, the variable $s$ is absent from the generating function $P(t,q)$ of pillars with at least one cell and thus becomes 
 \begin{equation*}
P(t,q)=\frac{tq}{1-tq}.
\end{equation*}
In the expansions of generating functions that follow, we will often omit to explicitly declare the formal variables $t$ and $q$ when their presence is obvious, especially for the generating functions $PB(t,q)$ and $B(t,q)$ that will simply be written $PB$ and $B$ respectively. 

%%%%%%%%%%%%%%%%%%%%%%%%%%%%%%%%%%%%%%%%%%%%%%%%%%%%
% \paragraph{\bf Inscribed bubbles} 
%%%%%%%%%%%%%%%%%%%%%%%%%%%%%%%%%%%%%%%%%%%%%%%%%%%%
\paragraph{\bf Bubbles of given width} The generating function for the bubbles $\mathcal{B}_2(k,n)$  of width two is
\begin{align*}
B_2(t,q)=\frac{t^3q^5}{1-tq}.
\end{align*}
The generating function of bubbles on top of pillars (see Figure~\ref{bubblesA})  is done by multiplication of generating functions. Thus a bubble of width $r$ with a pillar at its bottom, denoted $PB_r$, is the product  
 \begin{align*}
PB_r=B_r\times P
\end{align*}
of the corresponding generating functions. For instance $PB_2(t,q)=\frac{t^4q^6}{(1-tq)^2}$. Bubbles in $\mathcal{B}_3$ are obtained by juxtaposition of possibly several snakes in $\mathcal{PB}_2$ on top of  one bubble in $\mathcal{B}_2$.  Then two pillars are added on each end in the tail and head column together with two new cells on a new third column. These two new cells become the new head and tail of the bubble in $\mathcal{B}_3$. This construction, shown in Figure~\ref{bubblesB}, gives 
\begin{align*}
B_3(t,q)=\frac{B_2}{(1-PB_2)}\times \frac{q^2}{(1-tq)^2}=
\frac{t^3q^7}{((1-tq)^2-t^4q^6)(1-tq)}.
\end{align*}

To construct a $B_4$ snake, we need at least one $B_3$ to which we append an arbitrary number of $B_2$ and $B_3$ on each side. Assume that the mandatory  $B_3$ is the lowest among all the $B_3$ present (see Figure \ref{bubblesC_Cor}). Then below that $B_3$ is an arbitrary number of $B_2$ and above are an arbitrary number of $B_2$ and $B_3$.   Then we append a possible pillar and one cell in a new column on each extremity. All $B_4$ satisfy this description which is translated into the following  expression : 
\begin{align*}
B_4(t,q)&= \frac{B_3}{(1-PB_2)(1-PB_2-PB_3)} \times \frac{q^2}{(1-tq)^2}.
\end{align*}
\begin{figure}[htbp]
\begin{center}
\subfigure[$PB_2$]{\includegraphics[scale=0.8]{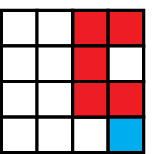}\label{bubblesA}}\qquad
\subfigure[$B_3$]{\includegraphics[scale=0.8]{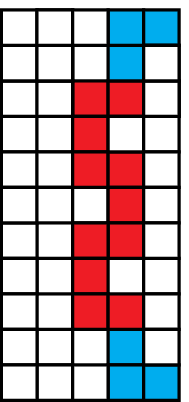}\label{bubblesB}}\qquad
\subfigure[$B_4$]{\includegraphics[scale=0.9]{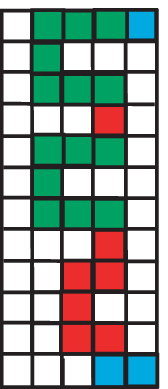}\label{bubblesC_Cor}}\qquad
\caption{Bubbles}
\label{bub}
\end{center}
\end{figure}

The precedent description for the set $\mathcal{B}_4$ can be extended to all sets $\mathcal{B}_r$ with $r\geq 3$ so that we have a rational generating function for each bubble in $\mathcal{B}_r$. 

\begin{proposition}\label{propbubl} For all integers $r\geq 3$ we have 
 \begin{align*}
B_r(t,q)&=\frac{B_{r-1}}{(1-\sum_{i=2}^{r-2}PB_i)(1-\sum_{i=2}^{r-1}PB_i)} \times \frac{q^2}{(1-tq)^2}.
\end{align*}
\end{proposition}

%%%%%%%%%%%%%%%%%%%%%%%%%%%%%%%%%%%%%%%%%%%%%%%%%%%%
% \paragraph{\bf  Non inscribed bubbles} 
%%%%%%%%%%%%%%%%%%%%%%%%%%%%%%%%%%%%%%%%%%%%%%%%%%%%
\paragraph{\bf Bubbles of arbitrary width} The task of finding a generating function $B(t,q,w)$ for bubbles of arbitrary width is made possible by the existence of  an elegant combinatorial factorization that appear in Figure~\ref{bubni} reminiscent of a factorization of Dyck paths (see for instance~\cite{Ge}). This figure informs us  that a bubble is either a minimal bubble or a stack of bubbles with pillars on top of an initial bubble terminated on each end by one possible pillar and a mandatory cell on a new column. 
\begin{figure}[htbp]
\begin{center}
\epsfig{file=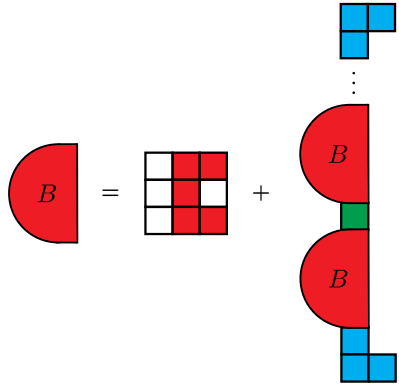,scale=1}
\caption{Combinatorial factorization of bubbles}
\label{bubni}
\end{center}
\end{figure}

This translates into the following equation where $B$ stands for $B(t,q,w)$. Recall that the variables $t$ and $q$ carry respectively the height and length statistics of the bubble. We introduce a new statistic called the horizontal half-perimeter which counts half the number of horizontal unitary cell edges on the perimeter of the bubble.  For example the bubble in Figure \ref{bubblesA} has horizontal half-perimeter equal to $3$. The half-perimeter statistic is known in the literature (see \cite{BmR}) and it is carried by the variable $w$ in $B$. 
\begin{align*}
B=&\frac{w^3t^3q^5}{1-tq} + \frac{wq}{1-tq} \frac{B}{1-PB} \frac{wq}{1-tq}.
\end{align*}
This equation  is easily solved to give the following expression:
\begin{align}\label{BNI}
& B=  \frac{1}{2tq(1-tq)} 
\left(1-2tq+w^3t^4q^6-w^2q^2+t^2q^2-\right.\\ \nonumber
&\left.\sqrt{(1+2wq -2tq + w^2q^2 +t^2q^2 -2wq^2t-w^3t^4q^6)(1-2wq-2tq +w^2q^2+t^2q^2+2wq^2t-w^3t^4q^6)}\right).
\end{align}
Setting $t=1$ and $w=1$ reduces equation ({\ref{BNI}}) to the length generating function $B(q)$:
\begin{align*}
%\label{BNIr}
 B(q)=  \frac{1}{2q} 
\left(1-q-q^2-q^3-q^4-q^5- \sqrt{(1+q)(1+q+q^2)(1-q+q^2)(1-q-q^2)(1-2q-q^3)}\right)
\end{align*}
This generating function, which is algebraic,  describes the sequence $A023422$ in the $OEIS$ bank of sequences. Bubbles snakes belong to the class of non-simple excursions in the lattice paths classification proposed in~\cite{BF}. This combinatorial factorization cannot handle the width parameter $r$. For that reason the generating function of bubbles of given width cannot be obtained with this method.\\
%%%%%%%%%%%%%%%%%%%%%%%%%%%%%%%%%%%%%%%%%%%%%%%%%%%%
% \paragraph{\bf Crossings} 
%%%%%%%%%%%%%%%%%%%%%%%%%%%%%%%%%%%%%%%%%%%%%%%%%%%%
\paragraph{\bf Crossings} The generating function of crossings of width $2$ and $3$ are easy to establish (see Figures~\ref{crossing2} and~\ref{crossing3}) and are given by the following expressions  
 \begin{align*}
  C_2(t,q)&=tq^2,\\
    C_3(t,q)&=\frac{tq^3}{1-tq}=\frac{q}{1-tq}C_2(t,q). 
\end{align*}

Crossings of width $b=4$ (see Figure \ref{crossing4}) are constructed from bottom to top and left to right as follow. First there is a crossing of width $3$. On top of it an arbitrary number (possibly zero) of bubbles of width two with their pillar are stacked 
and then a possible final pillar and finally one cell in the fourth column. This gives
  \begin{align*}
  C_4(t,q)&=C_3(t,q)\left( \frac{1}{1-PB_2}\right)\frac{q}{1-tq}.
  \end{align*}
Let's factor the crossing $C_5$ to describe its generating function and obtain a better understanding of the general case. Moving from bottom to top, the first factor is a crossing of width $4$ (green part in Figure~\ref{crossing5}). On top of that crossing is a stack of an arbitrary number of bubbles with pillars of width $2$ or $3$ (red part in Figure~\ref{crossing5}). Finally the end comes with a possible pillar on the fourth column and a last cell in the fifth column (blue part in Figure~\ref{crossing5}). This factorization gives the following expression: 
\begin{align*}
  C_5(t,q)&=C_4(t,q)\left( \frac{1}{1-PB_2-PB_3}\right)\frac{q}{1-tq}.
\end{align*}
\begin{figure}[htbp]
\begin{center}
\subfigure[$b=2$]{\makebox[1.5cm]{\includegraphics[scale=0.8]{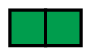}\label{crossing2}}}\qquad
\subfigure[$b=3$]{\makebox[1.5cm]{\includegraphics[scale=0.8]{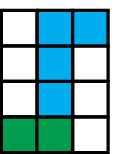}\label{crossing3}}}\qquad
\subfigure[$b=4$]{\makebox[1.5cm]{\includegraphics[scale=0.8]{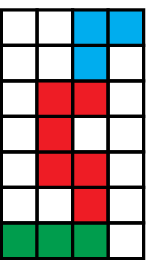}\label{crossing4}}}\qquad
\subfigure[$b=5$]{\includegraphics[scale=0.8]{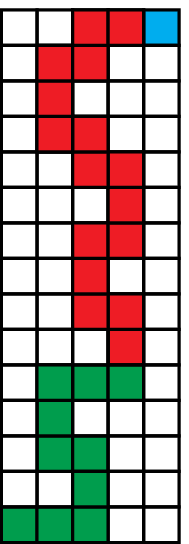}\label{crossing5}}\qquad
\caption{Crossings of width $b$}
\label{cros}
\end{center}
\end{figure}

For crossings of arbitrary width $b$, there is a similar factorization in three parts: $i)$ A crossing of width $b-1$ (green cells in Figure~\ref{cros}). $ii)$ A stack of bubbles of width $2$ to $b-2$ (red cells in Figure~\ref{cros}). $iii)$ A pillar in column $b-1$ and one cell in column $b$ (blue cells in Figure~\ref{cros}). We have proved the following result. 
\begin{proposition} \label{propcros1}For all $b\geq 3$ we have 
 \begin{align*}
C_b(t,q)&=C_{b-1}(t,q)\left( 
\frac{1} {1-\sum_{i=2}^{b-2}PB_i}\right) \frac{q}{1-tq}
&=   \left(\frac{q}{1-tq}\right)^{b-2}\frac{tq^2}{\prod_{j=2}^{b-2}\left( 1-\sum_{i=2}^{j}PB_i\right)}.
\end{align*}
\end{proposition}
From propositions \ref{propbubl} and \ref{propcros1}  we obtain an expression for crossings in terms of bubbles :
\begin{proposition} For all integers $b\geq 3$ we have 

\begin{align*}
C_{2k+1}(t,q)=\frac{\prod_{i=1..k}B_{2i}} {\prod_{i=2..k}B_{2i-1}}\times \frac{1}{t^2q^2}\\
C_{2k}(t,q)=\frac{\prod_{i=2..k}B_{2i-1}} {\prod_{i=1..k-1}B_{2i}}\times \frac{(1-tq)} {t^2q^3}
\end{align*}
\end{proposition}
%%%%%%%%%%%%%%%%%%%%%%%%%%%%%%%%%%%%%%%%%%%%%%%%%%%%%
% \paragraph{\bf Endings}
%%%%%%%%%%%%%%%%%%%%%%%%%%%%%%%%%%%%%%%%%%%%%%%%%%%%%
\paragraph{\bf Endings} The most elementary non bubble or non crossing end of a $PDS$ is a pillar. To fix ideas, let's number the columns of the rectangle from left to right. As a warm up, we construct the generating function $E_b^2(t,q)$ of ending  snakes in a rectangle of width $b$ with the last cell on the second column (see Figure~\ref{famEnd}) by stacking its parts upwards. 

  We start with a crossing of width two (green cells in Figure~\ref{endA}). The second part is a possible pillar (blue cells in Figure~\ref{endA}). The third part is a possible stack of bubbles (red cells in Figure~\ref{endA}) of variable width of value at most $b-2$.  But the first cell of the first bubble can be the last cell of the initial crossing which means that we divide the generating function of the first bubble by $tq$. Finally the fourth part is a possible final pillar on top of the last bubble (last blue cell in Figure~\ref{endA}). With this factorization we obtain
 \begin{align*}
 E_b^2(t,q)
  &=tq^2 \times \frac{1}{1-tq}\times \left( 1+ \frac{\sum_{i=2}^{b-2}B_i} {tq(1-\sum_{i=2}^{b-2}PB_i)}
  \times \frac{1}{1-tq}\right).
  \end{align*}
  
Now lets turn to the general case $E_b^c(t,q)$ where the ending finishes on column $c$ with $1<c<b$.  We are going to factor the snake from left to right and bottom to top. As in $E_2$, we start with a horizontal row of two  cells (green cells in Figure~\ref{endB}). The next factor is the move from column two to column three  made of a possible pillar 
(first blue pillar in Figure~\ref{endB})  followed by one of the two possible patterns: $i)$ A unique cell in the third column $ii)$ At least one bubble on the second column followed by a pillar of height at least two and a cell in the third column. The first bubble may start with the last cell of the previous factor. 
 This gives 
  \begin{align}
  \label{fact1}
 \mbox{First factor:}\; \frac{tq^2}{1-tq} \left( q+\frac{\sum_{i=2}^{b-2}B_i} {(1-\sum_{i=2}^{b-2}PB_i)}\frac{t^2q^3}{tq(1-tq)}
  \right).
  \end{align}
 
 The second factor starts at the end of the first factor and  ends with the first cell on the fourth column immediately after the last visit on the third column.
It describes the part of the snake from column three to column four, as shown in Figure~\ref{endC}, where the black cells belong to the previous factor. 
This gives 
\begin{align}
\label{fact2}
 \mbox{Second factor:}\; \frac{1}{1-tq} \left(q+\frac{\sum_{i=2}^{b-3}B_i} {(1-\sum_{i=2}^{b-3}PB_i)}\frac{t^2q^3}{tq(1-tq)}
  \right).
  \end{align}
 
Each factor is constructed similarly, describing the move from one column to the next. There is a total of $c-1$ factors. The last factor is slightly different due to the fact that the snake does not move to the next column but starts and terminates on column $c$, as in Figure~\ref{endD} where the black cells belong to the previous factor. 
 \begin{align}
 \label{factc}
 \mbox{Last factor:}\; \left( 1+\frac{\sum_{i=2}^{b-c}B_i} {(1-\sum_{i=2}^{b-c}PB_i)tq(1-tq)} \right)
 \frac{1}{1-tq}.
  \end{align}
 \begin{figure}[h!]
\begin{center}
\subfigure[$E_b^2$]{\includegraphics[scale=0.73]{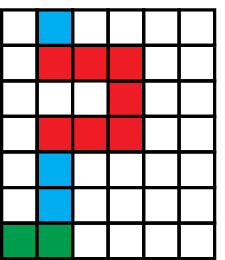}\label{endA}}\qquad
\subfigure[First factor of $E_b^c$]{\includegraphics[scale=0.73]{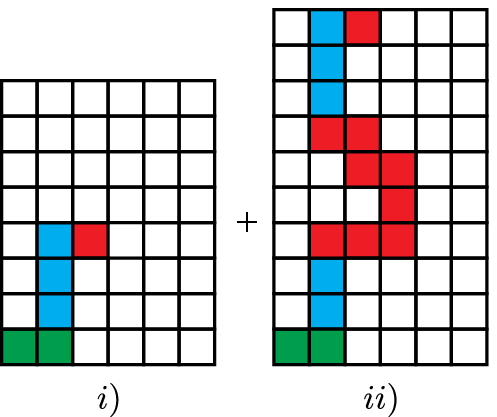}\label{endB}}\qquad
\subfigure[Second factor of $E_b^c$]{\includegraphics[scale=0.73]{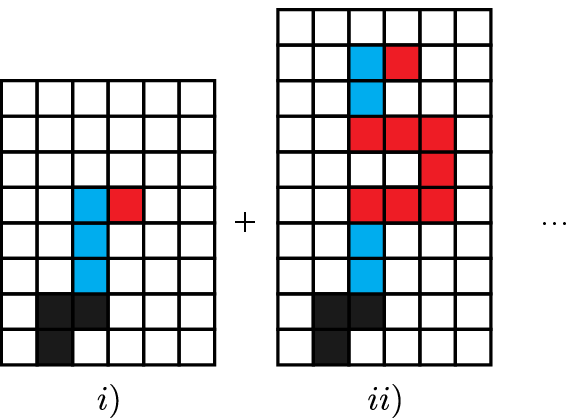}\label{endC}}\quad
\subfigure[Last factor of $E_b^c$]{\includegraphics[scale=0.73]{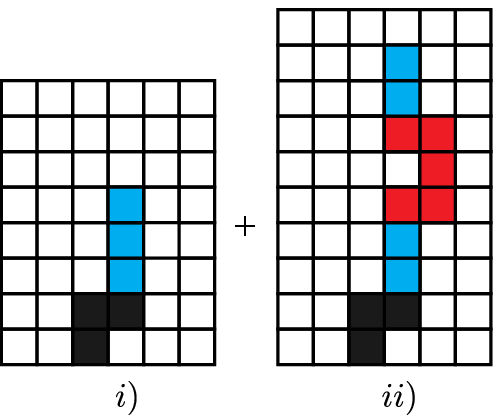}\label{endD}}
\caption{Endings}
\label{end}
\end{center}
\end{figure}
\allowdisplaybreaks 
From equations~\eqref{fact1}, \eqref{fact2} and~\eqref{factc} and the multiplication principle, we deduce the following expression.
\begin{proposition} For all integers $b$ and $c$ such that $1<c<b$, the generating function for the ending snakes finishing on column c in a  rectangle of width $b$ is the following  
\begin{align*}
E_b^c(t,q)&=\frac{tq^c}{(1-tq)^{c-1}}\times  \frac{1}{\prod_{j=0}^{c-3} \left(1-\sum_{i=2}^{b-2-j}PB_i\right)}\times 
\left(1+ \frac{\sum_{i=2}^{b-c}PB_i}{t^2q^2\left(1-\sum_{i=2}^{b-c}PB_i\right)} \right). 
\end{align*}
\end{proposition}
Finally the generating function $E_b(t,q)$ of the endings in a rectangle of width $b$ is the sum of the generating functions $E_b^c(t,q)$ for all columns $c$ from $c=2$ to $c=b-1$. The case $c=1$ being a pillar, we have $E_1(t,q)=\frac{1}{1-tq}$ which will be incorporated later in the families $\mathcal{F}_i$.
\begin{proposition}  The generating function $E_b(t,q)$ for the ending snakes  in a $b\times k$ rectangle is the following:
\begin{align*}
E_b(t,q)&=\sum_{c=2}^{b-1} E_b^c(t,q)\\
&= \sum_{c=2}^{b-1} \frac{tq^c}{(1-tq)^{c-1}}\times  \frac{1}{\prod_{j=0}^{c-3} \left(1-\sum_{i=2}^{b-2-j}PB_i\right)}\times 
\left(1+ \frac{\sum_{i=2}^{b-c}PB_i}{t^2q^2\left(1-\sum_{i=2}^{b-c}PB_i\right)} \right).
\end{align*}
\end{proposition}
\paragraph{\bf Generating functions for inscribed {\em PDS}}
We are now in a position to present  three-variable generating functions $F_i(s,t,q)$ for the families of inscribed $PDS$ with the structure given by equations~\eqref{f1} to~\eqref{f4}.  We obtain
\begin{align*}\allowdisplaybreaks
F_1(s,t,q)&=2\sum_{b\geq 2} s^b \left[ [1+P(t,q)(1+E_b(t,q))]^2\times \frac{ P(t,q)C_b(t,q)^2}{\left[1-P(t,q)(C_b(t,q)+\sum_{i=2}^{b-1}B_i(t,q))\right]}
\right]\\
F_2(s,t,q)&=4 \sum_{b\geq 2} s^b \left[ [1+P(t,q)(1+E_b(t,q))]^2\times 
\frac{ P(t,q)C_b(t,q)\sum_{i=2}^{b-1}B_i(t,q)}{\left[1-P(t,q)(C_b(t,q)+\sum_{i=2}^{b-1}B_i(t,q))\right]}
\right]\\
F_3(s,t,q)&=2\sum_{b\geq 2} s^b \left[ [1+P(t,q)(1+E_b(t,q))]^2\times 
\frac{ P(t,q)\sum_{i=2}^{b-1}B_i(t,q)}{\left[1-P(t,q)\sum_{i=2}^{b-1}B_i(t,q)\right]}\times \right.\\
 &\hspace{6.5cm}\left.\frac{ P(t,q)C_b(t,q)\sum_{i=2}^{b-1}B_i(t,q)}{\left[1-P(t,q)(C_b(t,q)+\sum_{i=2}^{b-1}B_i(t,q))\right]}
\right]\notag\\
F_4(s,t,q)&=\sum_{b\geq 2} 2s^b \left[ [1+P(t,q)(1+E_b(t,q))]^2\times C_b(t,q)
\right] -tq^b
\end{align*}

These four last expressions are not rational but, together with Proposition~\ref{proppart}, they allow the production of a two-variable rational expression for each value of the width $b$.

%=============================================================
% \section{A Bijection} 
%=============================================================

\section{A Bijection}
\label{bij}
	Bargraphs are a class of polyominoes that has received the interest of combinatorists who provided their generating functions according to different parameters  (see~\cite{BmR}).  Bargraphs are defined as column-convex polyominoes with their lower edge lying on the horizontal axis. They are uniquely characterized by the height of their columns.  Looking at bargraphs and bubble snakes simultaneously, we are compelled to guess that bubbles are contours of bargraphs.  Indeed there exists a bijection between these two sets of objects that we describe in  the present section. 

In order to uniformize our presentation, we apply on bargraphs a counterclockwise rotation of 90 degrees. So from now on, we adapt our notation to the rotated bargraphs. A bargraph has width $h$ when  its longest horizontal row has length $h$ (see figure \ref{Bargraph}). Given a bargraph, the corresponding bubble will be the set of cells on the perimeter of the bargraph with  side or corner contact with the restriction that the bubble has no cell on the right side of the bargraph. We can thus see a bubble as the envelop of a bargraph. But before envelopping a bargraph, we need to stretch its wells. 
	
A $well$ of height $k$ is a part of a bargraph formed by two rows of length $h_0$ and $h_2$ such that all the $k$ rows between them are of the same length $h_1$ with $h_0> h_1 < h_2$. For example the bargraph in  figure~\ref{Bargraph} has a well of height $k=2$.
Bargraphs having wells of height $k< 3$  cannot be directly enveloped in a bubble.  So we increase  by two the height of each well in a bargraph to guarantee that all wells have height at least $3$ and then we envelop the bargraph with a unique bubble as explained above.
\begin{proposition}
	\upshape There is a bijection $f: bargraph(h-1) \rightarrow bubble(h)$ from the set of bargraphs of width $h-1$ to the set of bubbles of width $h$. 
\end{proposition}
\begin{proof}
Let us consider a bubble. The bargraph sent to this bubble is the interior of the latter from which we remove two rows in each well. So $f$ is surjective. Given a bargraph the corresponding bubble is the envelop of the bargraph after the addition of two rows in each well (see figures~\ref{Sbargraph} and~\ref{Bubble}). The two-step construction of $f$ garantees that  distinct bargraphs are sent by $f$ to distinct bubbles
so that $f$ is injective and thus bijective.
   \end{proof}
\begin{figure}[h!]\centering
	\subfigure[Bargraph]{\makebox[3cm]{\includegraphics[scale=0.8]{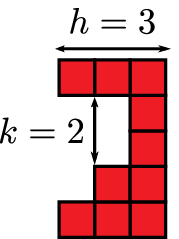} \label{Bargraph}}}\quad
	\subfigure[Stretched bargraph]{\includegraphics[scale=0.8]{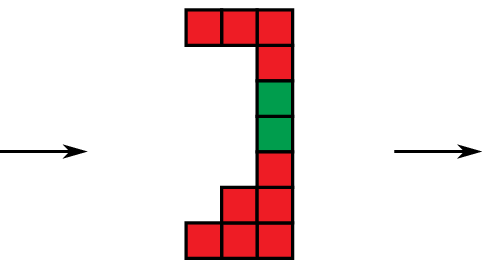} \label{Sbargraph}}
	\subfigure[Bubble]{\makebox[2.6cm]{\includegraphics[scale=0.8]{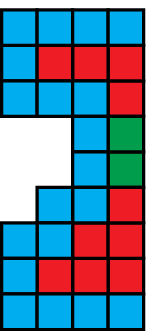} \label{Bubble}}}
	\caption{Bijection $f:~bargraph(h-1) \rightarrow bubble(h)$}
\end{figure}

%=============================================================
%\section{Inscribed snakes of maximal length}
%=============================================================

\section{Inscribed snakes of maximal length}
\label{conj}
In the course of our investigations of snake polyominoes, we have been interested in inscribed snakes of maximal length. To this end, partially directed snakes and kiss-free snakes have been studied. We define a {\em  kiss-free snake} as a snake polyomino  
such that there exists no  $2\times 2$ rectangle with two cells of the snake in one of its diagonals and two empty cells in its other diagonal (see Figure \ref{Maximal_kissfree}). 

It is straightforward to obtain the value of the maximal length of $PDS$ inscribed in a rectangle of size 
$b\times k$ which is 
$ \left\lfloor \frac{(b+1)(k+1)}{2}\right\rfloor-1$, as illustrated on figure \ref{Maximal_PDS1}. 
In an earlier version of this text, we conjectured that this length is also the maximal length for any inscribed kiss-free snake, as shown on figure \ref{Maximal_kissfree}. This fact which is the content of the following conjecture has recently been generalized to tree polyominoes and proved by a different group of authors  (see \cite{BGL}).
\begin{conjecture}
The maximal length of kiss-free snakes inscribed in a $b\times k$ rectangle is
\begin{align}\label{maxl}
 \left\lfloor \frac{(b+1)(k+1)}{2}\right\rfloor-1.
\end{align}
\end{conjecture} 
\begin{figure}[h!]
\centering
\subfigure[{\em PDS}]{\includegraphics[scale=0.8]{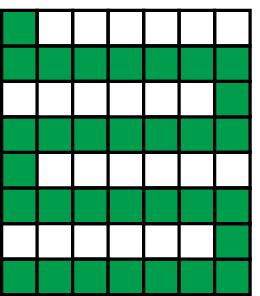}\label{Maximal_PDS1}}\qquad
\subfigure[kiss-free]{\includegraphics[scale=0.8]{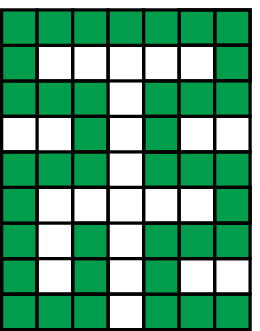}\label{Maximal_kissfree}}\qquad
\subfigure[general snake]{\makebox[2.5cm]{\includegraphics[scale=0.8] {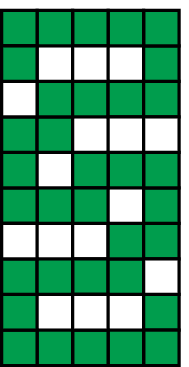}\label{Max_uniqB}}}
\caption{Inscribed snakes of maximal length}
\end{figure}

In other words, when we relax the condition of being $PDS$, there is no gain when the goal is to fill a rectangle with the longest snake as long as the snake remains kiss-free. For arbitrary  snakes this conjecture is not true: figure~\ref{Max_uniqB} shows a snake with length exceeding expression~\eqref{maxl}. We could not conjecture an exact expression for 
the maximal length of inscribed arbitrary snakes in terms of $b$ and $k$. 
A. Goupil's work  was supported by  FODAR.  M.-E. Pellerin is grateful to NSERC for a graduate scholarship. The three authors thank H. Lessard for the conception and implementation of  computer programs that generate inscribed snakes.

%%=============================================================
%% Bibliography
%%=============================================================

\end{document}